\documentclass[12pt]{amsart}
\usepackage{lineno,hyperref}
\modulolinenumbers[5]
\usepackage{comment}
\usepackage[english]{babel}
\usepackage[utf8]{inputenc}
\usepackage{amsmath}
\usepackage{shuffle}
 \usepackage{graphicx}
\usepackage[colorinlistoftodos]{todonotes}
\usepackage{amsthm}
\usepackage{amsfonts}
\usepackage{upgreek}
\usepackage[ruled,vlined]{algorithm2e}
\usepackage{float}
\usepackage{color}
\usepackage{xcolor}
\usepackage[all]{xy}
\usepackage{subfig}
\usepackage{tabularx}
\usepackage{amssymb}
\usepackage{resizegather}
\usepackage{color}
\usepackage{amsmath}
\usepackage{tikz-cd}
\usepackage{qtree}
\usepackage{graphicx}
\usepackage{url}

\usetikzlibrary{matrix, calc, arrows}
\usepackage{tikz}
\usetikzlibrary{positioning}
\usepackage{tkz-euclide}
\usepackage[foot]{amsaddr}
\usepackage{fullpage}
\usepackage{array}
\usepackage{mathabx}%for maya caracters

\usepackage{adjustbox}
\newtheorem{theorem}{Theorem}[section]
\newtheorem*{theorem*}{Theorem}
\newtheorem{lemma}[theorem]{Lemma}
\newtheorem*{lemma*}{Lemma}

\newtheorem{corollary}[theorem]{Corollary}

\newtheorem*{corollary*}{Corollary}

\theoremstyle{definition}
\newtheorem{definition}[theorem]{Definition}
\newtheorem{example}[theorem]{Example}
\newtheorem{remark}[theorem]{Remark}

\newtheorem*{exercise*}{Exercise}

\usepackage{mathtools}
\def\multiset#1#2{\ensuremath{\left(\kern-.3em\left(\genfrac{}{}{0pt}{}{#1}{#2}\right)\kern-.3em\right)}}

\newcommand{\chain}[1][n]{\langle #1\rangle}

\newcommand{\nsomega}[2][n]{\Omega(#2,#1)}
\newcommand{\somega}[2][n]{\Omega^{\circ}(#2,#1)}

\newcommand{\zetaf}[1][X]{\mathfrak{Z}\left(#1\right)}
\newcommand{\zetafp}[1][X]{\mathfrak{Z}^+\left(#1\right)}

\newcommand{\zetafs}[1][X]{\mathfrak{Z}^{s}\left(#1\right)}

\usetikzlibrary{decorations.markings}
\tikzstyle{vertex}=[circle, draw, inner sep=0pt, minimum size=6pt]

\title{Shuffle Series}

    \author[1]{{Khushdil} {Ahmad$^1$}
    }
    
\author[2]{{Eric R}
 {Dolores-Cuenca}$^2{}^4*$}

\author[3]{{Khurram} {Shabbir$^3$}}

\email{$\ast$eric.rubiel@pusan.ac.kr}
\address{$^1{}^4${
{Department of Mathematics}, {Government College University}, {Lahore}, {Pakistan}}}
\address{$^2${Industrial Mathematics Center}, {Pusan National University}, { {Busan}, {Korea}}}\address{$^3${Department of Mathematics}, {Yonsei University}, { {Seoul}, {Korea}}}
\begin{document}

\begin{abstract} 
We apply operad theory to enumerative combinatorics in order to count the number of shuffles between series-parallel posets and chains. 
We work with three types of shuffles, two of them noncommutative, for example a left deck-divider shuffle $A$ between $P$ and $Q$ is a shuffle of the posets in which, on every maximal chain $m\subset A$, the minimum and maximum elements belong to $P$ and no two consecutive points of $Q$ appear consecutively on $m$. The number of left deck-divider shuffles of $P$ and $Q$ differ from the number of left deck-divider shuffles of $Q$ and $P$. 

The generating functions whose $n$ coefficient counts shuffles between a poset $P$ and $1<2<\cdots<n$ are called shuffle series. We explain how shuffle series are isomorphic to order series as algebras over the operad of series parallel posets.

The weak and strict order polynomials are well known in the literature. At the level of generating series, with the theory of sets with a negative number of elements, we introduce a third order series and prove a theorem in the style of Stanley's Reciprocity Theorem compatible with the structure of algebras over the operad of finite posets. 

We conclude by describing the relationship of our work with the combinatorial properties of the operadic tensor product of free trees operads.
\end{abstract}

\maketitle
\section{Introduction}

Given two decks, one with $n$ cards $A_1<\cdots<A_n$ and one with $m$ cards $B_{1}<\cdots<B_{m}$, the deck obtained from giving a linear order $<_{Sh}$ to $\{A_1,\cdots,A_{n}\}\cup\{B_1,\cdots,B_{m}\},$  is called a $(n, m)$-shuffle if the image of the original decks preserve their own order $A_1<_{Sh}\cdots<_{Sh}A_{n},$ and $ B_{1}<_{Sh}\cdots<_{Sh}B_{m}$.
The shuffle of decks appears classically in probability and combinatorics, and it is used to define a shuffle product of multizeta values in number theory~\cite{ EulerSh, shuffle}. Recent work on operad theory~\cite{SandDHT, Tensor} motivates an extension of shuffles of decks to shuffles of other partially ordered sets (posets). We will provide a precise definition of a shuffle between a poset and a $n$-chain $\chain[n]=1<2<\cdots<n$ in Section~\ref{Sec:First}.

A finite poset is a set with a partial order. We can also describe a poset as a finite set with a particular Alexandrov topology, then continuous functions coincide with order preserving functions.
By a series-parallel posets (${SP}$-poset), we mean a finite poset generated by a point under the operations of ordinal sum and disjoint union. Given a ${SP}$-poset $P$, we study the generating series $\sum_{i=0}^\infty a_ix^i$ whose $n$ coefficient $a_n$ is the number of shuffles between $P$ and $\chain[n]$. We call these generating series shuffle series. We denote the shuffle series of $P$ by $\mathcal{SH}(P)$, and we denote $\mathcal{SH}(i)=\mathcal{SH}(\chain[i])$. The shuffle series of a point is $\mathcal{SH}(1)=\frac{1}{(1-x)^2}$. 

This paper is part of a larger project, developing the theory of algebras over the operad of finite posets. The language of operads allows us to study customized operations and the sets in which those operations are defined. The paper~\cite{OpTre}  foresaw transparency as one of the advantages of understanding the underlying operadic structure of classical objects in combinatorics. As in the case of~\cite{Monops} in which posets and operads provide a combinatorial explanation of the inverse of Riordan matrices.
%The operad is topological
A symmetric topological operad $\mathcal{O}$ consists of topological spaces of $n$-ary operations $\{\mathcal{O}(n)\}_{n\in\mathbb{N}}$, with unital associative continuous composition rules and actions of the group of permutations. An algebra over the operad $\mathcal{O}$, an $\mathcal{O}$-algebra for short, is a set in which we realize the elements of $\mathcal{O}$ as associative operations. 

\begin{definition}
    Let $P=(\{x_i\}_{1\leq i\leq n},\leq_P)$ be a finite poset with $|P|=n$, and consider finite posets $\{Q_i\}_{1\leq i\leq n}$. The lexicographic sum of posets~\cite{Doppelgangers2,osets,order}  $\widetilde{P}(Q_1,\cdots,Q_n)$ is the poset with points $\sqcup_i\{q\in Q_i\}$ with the order:
$$a\leq_{P(Q_1,\cdots,Q_n)} b \hbox{ if } \begin{cases}
a\leq_{Q_i}b, &\hbox{for some } i, \\
a\in Q_i, b\in Q_j &\hbox{ and } x_i\leq_P x_j.
\end{cases}$$\label{LS}
\end{definition}

The first example of a topological operad is given by the operad of finite posets $\mathcal{FP}$ with the lexicographic sum as operadic composition. 
The automorphisms of the operad $\mathcal{FP}$ are of topological nature, we will expand on this point in the paper. We use different notation to help the reader distinguish when $P$ is a poset or when the same poset $\widetilde{P}$ is defining an $n$-ary operation, $ n=|P|$. 

Series parallel posets form a suboperad $\mathcal{SP} \subset \mathcal{FP}$. As a first introduction, we consider the generators of $\mathcal{SP}$ as the operations $\widetilde{\{c<d\}}$ and $\widetilde{\{c, d\}}$. 
We show that the set of shuffle series $\{\mathcal{SH}(P)\}_{P\in SP-\hbox{posets}}$ possess the structure of an algebra over the operad $\mathcal{SP}$.
We use this algebraic structure to systematically solve enumeration problems of posets.

Connecting two fields, enumerative combinatorics, and operad theory, implies that our main results either describe the hidden structure of a set or they provide a combinatorial interpretation of the objects under study.

Our first result from the structural point of view explains how to compute the action of  $\widetilde{P}\in \mathcal{SP}(n)$, on  $(\mathcal{SH}(1),\cdots,\mathcal{SH}(1))$ (assuming $n$ entries). We define a factorization of a poset $P$ to be a tree with indecomposable posets as internal vertices, where the grafting of trees acts as the lexicographic sum, and after applying the lexicographic sum indicated by the tree the result should be $P$. Fix a factorization of $P$.

\begin{itemize}
\item If the action of the indecomposable posets are endomoprhisms of $\mathbb{Z}[\{\mathcal{SH}(i)\}_{i\in\mathbb{N}}]$.

\item We compute the series solving the enumerating problem for a point $\mathcal{SH}(1)=\frac{1}{(1-x)^2}$.
\item  Assuming that the leaves of the tree contain drops of water, each drop is a copy of $\mathcal{SH}(1)$. As the drops fall along the tree, we apply the corresponding operation until the drops reach the root.   
\item  The value at the root is $\widetilde{P}(\mathcal{SH}(1),\cdots,\mathcal{SH}(1))$. 
\end{itemize}

The corresponding combinatorial theorem explains that the $k$ coefficient of the series constructed counts shuffles between $P$ and $\chain[k]$. In other words, the shuffle series $\mathcal{SH}(P)$ is the series $\widetilde{P}(\mathcal{SH}(1),\cdots,\mathcal{SH}(1))$. We introduce two conditions in Lemma~\ref{Lemma:conditions} to formalize and generalize the previous strategy. 

Our results imply that shuffle series of $SP$-posets are the $\mathcal{SP}$-algebra generated by $\frac{1}{(1-x)^2}$, in which the action of $\widetilde{\{c, d\}}$ on generating series $f,g$ is the Hadamard product of series,
and the action of $\widetilde{\{c<d\}}$ is the deformed product $f,g\mapsto f(1-x)g$. 

A map $f$ is called strictly order-preserving if $x<y$ implies $f(x)<f(y)$, and the map is called weakly order-preserving if $x\leq y$ implies $f(x)\leq f(y)$. There is a second algebra %$\{\zetafp[P]\}_{P\in \hbox{Posets}}$ 
over the operad $\mathcal{SP}$ called weak order series. The $n$ coefficient of $\zetafp[P]$ the weak order series of a poset $P$, counts the number of weak order-preserving maps from $P$ to $\chain[n]$, and the weak order series of a point is $\zetafp[1]=\frac{x}{(1-x)^2}$.  Order series of $SP$-posets are studied in~\cite{posets} as an algebra over the operad $\mathcal{SP}$, and it is proved that for a fixed poset $P$ to know the number of strict-order preserving maps to every chain is equivalent as to know the number of weak-order preserving maps to every chain. A shuffle between two chains $\chain[n]$ and $\chain[m]$ is the same as an strict order-preserving map from $\chain[n]$ to $\chain[n+m]$, from this intuition we prove that the operadic map $\frac{1}{(1-x)^2}\rightarrow \frac{x}{(1-x)^2}$ induces an isomorphism of algebras over the operad $\mathcal{SP}$ between shuffle series and weak order series. 

We now work with graphs, specifically with reduced trees. A reduced tree $T$ is the result of pruning all the leaves of a tree and then putting exactly one leaf at each top vertex. When we restrict our work to posets that can be realized as the poset of vertices of a reduced tree, we answer the combinatorial question requested in~\cite[Page 64]{ShuffleTree}: What is the number of shuffles between a Tree and a linear Tree? We learned about this problem in the first of the 2023 Haifa distinguished lecture series by Ieke Moerdijk, at Utrecht University.
While the original problem is established on the level of polynomials, we believe the change of variable obtained by working on generating series is more algebraically friendly due to the fact that the Cauchy product of generating functions encodes a cumbersome product of polynomials. 

Once solved the question that motivated the project, Section~\ref{Sec:Second} applies enumerative combinatorics theory to understand the shuffle of posets. The order polynomial and the weak order polynomial are classic topics in enumerative combinatorics~\cite{EC}. At the level of generating functions we introduce a third order series using the theory of sets with negative number of elements. As a consequence, we construct three generalizations of shuffles on linear orders to shuffles of posets: colimit-indexing shuffles, right deck-divider shuffles and left deck-divider shuffles.

The paper is organized as follows. Section~\ref{Sec:First} explains how to compute shuffle series of series-parallel posets. 
The Theorem~\ref{Thm:equality} gives an explicit description of the action of the operad $\mathcal{SP}$ on the shuffle series. In Section~\ref{TS} we explain how to compute the number of shuffles of trees with linear trees.

The Section~\ref{Sec:Second} applies ideas from Loeb and Ehrhart to the study of shuffles. Firstly, Lemma~\ref{Lemma:CM} and Corollary~\ref{Cor:basis} are analogues to Stanley Reciprocity Theorem, they introduce a third order series derived from weak order series and provide a combinatorial meaning to the series. 
Then, we explain in Theorem~\ref{main:thrm} how the enumeration of shuffles is equivalent to the enumeration of order-preserving maps. We study three families of shuffles, defined in Lemma~\ref{lemma:3sh}.
Finally, in Remark~\ref{RM:vectors} we list all the interpretations of the vectors that we compute, including those related to the combinatorics of the tensor product of tree operads.

In the Section~\ref{Sec:Third}, we indicate open problems.

\section{Shuffle series}
\label{Sec:First}
We work with finite posets. A $n$-chain is the poset $\chain[n]=\{1<2<\cdots<n\}$. The disjoint union of posets $P, Q$ is denoted by $\widetilde{\{c,d\}}(P, Q)$, it is obtained by the union of the underlying points, and no new order relations besides the original ones. The ordinal sum of posets $P$ and $Q$  $\widetilde{\{c<d\}}(P, Q)$ adds the following relations to the union of the posets: we require the 
maximum elements of $P$ to be smaller than the minimum elements of $Q$. For example, the ordinal sum of two chains is a chain $\widetilde{\{c<d\}}(\chain[k],\chain[l])=\chain[k+l]$. The Hasse diagram of a poset is a graph with nodes the points of the poset, and the edges occur between an element and its successor. If $y$ is a successor of $x$ then we draw a vertical edge from the point $x$ located at the bottom to the point $y$ located at the top. For example, Figure~\ref{fig:Ex1} shows the Hasse diagram of $\widetilde{\{c<d\}}(\widetilde{\{c,d\}}(\chain[1],\chain[1]),\widetilde{\{c,d\}}(\chain[1],\chain[1]))=\{r<s,r<t, q<s,q<t\}$.

\begin{figure}
    \centering
    \begin{tikzpicture}[scale=0.2]
\tikzstyle{every node}+=[inner sep=0pt]
%\tikzstyle{every node}[inner sep=0pt]
\draw [black] (5,-23.8) circle (3);
\draw [black] (5,-4.2) circle (3);
\draw [black] (15.7,-23.8) circle (3);
\draw [black] (15.7,-4.2) circle (3);
\draw [black] (15.7,-20.8) -- (15.7,-7.2);
\fill [black] (15.7,-7.2) -- (15.2,-8) -- (16.2,-8);
\draw [black] (5,-20.8) -- (5,-7.2);
\fill [black] (5,-7.2) -- (4.5,-8) -- (5.5,-8);
\draw [black] (14.26,-21.17) -- (6.44,-6.83);
\fill [black] (6.44,-6.83) -- (6.38,-7.77) -- (7.26,-7.3);
\draw [black] (6.44,-21.17) -- (14.26,-6.83);
\fill [black] (14.26,-6.83) -- (13.44,-7.3) -- (14.32,-7.77);
\end{tikzpicture}
    \caption{The Hasse diagram of the poset $\{r<s,r<t, q<s,q<t\}$.}
    \label{fig:Ex1}
\end{figure}
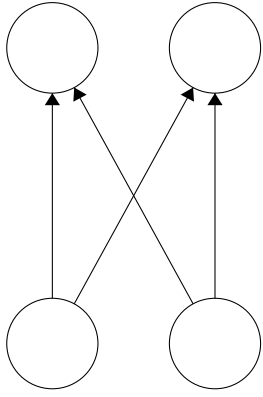
%%%%%%%%%%%%%%%%%%%%%%%%%%%%%%%%%%%%%%%%%%%%%%%%%%%%%%%%%%%%%%%%%%%%%%%%%%%%%%%%%%%%%%%%%%%%%%%%%%%%%%%%%%%%
A series-parallel poset (${SP}$-Poset) is a poset generated by the operations of disjoint union and ordinal sum applied on chains.
A binary tree is a finite acyclic graph, whose external edges are connected to one vertex only, every vertex has two incoming edges and one outgoing edge, and one external edge is called the root while the other external edges are called leaves.

%add definition
\begin{definition}
Let $P$ be a $SP$-poset with $|P|=n$. A factorization $\widetilde{P}$ of $P$ is a binary tree with every vertex labeled by $\widetilde{\{c<d\}}$ or $\widetilde{\{c, d\}}$.

To understand the evaluation of $\widetilde{P}$ on the vector with $n$ coordinates $(\chain[1],\cdots,\chain[1])$
we imagine each $\chain[1]$ as a drop of water, one on each leaf. As gravity pulls the drops along the branches of the tree, at the branching points we apply the operation that labels the vertex. We proceed until we reach the root, which should coincide with $P$.
\end{definition}

In computer science, there is a method to obtain the factorization of a poset. First, one computes the Reverse Polish notation (RPN)~\cite{RPN} of the poset. Then, the corresponding abstract syntax tree is the factorization of $P$.

For example, the poset $\{r<s,r<t, q<s,q<t\}$ has Hasse diagram in Figure~\ref{fig:Ex1}. It follows that one RPN associated with the poset is \[r\, q\, \{c,d\}\, s\, t\, \{c,d\}\, \{c<d\}.\]

\begin{lemma}\label{SP-N} The poset $P$ is a ${SP}$-poset if and only if the poset does not contain the $N$ poset: $ \{m<l>n<k\}$\end{lemma}
\begin{proof}
    See~\cite{SPPoset}.
\end{proof}
We reformulate Lemma~\ref{SP-N} in the following Lemma:
\begin{lemma}In an ${SP}$-poset $P$, if $y\in Successor(x)$ then for every $w\in Antecessor(y)$ we obtain $Successor(w)=Successor(x)$.\label{Lemma:N}
\end{lemma}\begin{proof}

If $P$ is an $SP$-poset, and $x,y\in P$  with $y\in Successor(x)$, and there is an $w\in Antecessor(y)$ with $Successor(w)\neq Successor(x)$, then one can find a $N$ poset in $P$, which is a contradiction according to Lemma~\ref{SP-N}.  
\end{proof}
From Lemma~\ref{Lemma:N}, when we cluster elements of the poset that share one successor, the elements of the cluster share all successors. This process can help us find a factorization of an ${SP}$-poset. After one iteration, elements of the cluster represent iterations of the operation $\widetilde{\{c,d\}}$ on points of the poset. The result is the first input of the operation $\widetilde{\{c<d\}}$. The second input is the union of those successors shared by the elements of the first input. A similar approach is applied if several elements have the same antecessor. After perhaps iterations of this process one can try to read the RPN.  The first iteration cluster posets. The second iteration clusters sets of sets of posets, and so on. In the case of the poset $\{r<s,r<t, q<s,q<t\}$
we obtain $\widetilde{\{c<d\}}\left(\widetilde{\{c,d\}}(r,q),\widetilde{\{c,d\}}(s,t)\right)$. The Hasse diagram is an alternative visual aid that shows how to group successors and group antecessors.
  \begin{figure}[htb]\centering
\begin{tikzpicture}[xscale=1.3,yscale=1.3]
 
\draw[blue,fill] (1.5,1.5) circle [radius=1pt]
node[right] {$\{c,d\}$};
\draw[blue,fill] (0.5,2.5) circle [radius=1pt]
node[right] {};
\draw[blue,fill] (2.5,2.5) circle [radius=1pt]
node[right] {};
\draw (.1,.1)  to  (1.4,1.4);
\draw (1.4,1.6)  to  (.6,2.4);
\draw (1.6,1.6) to  (2.4,2.4);

\draw[blue,fill] (0,0) circle [radius=1pt]
node[right] {$\{c<d\}$};
\draw[blue,fill] (0,-1) circle [radius=1pt];
\node[left] at (0.1,-0.5) {};

\draw (-1.4,1.4)  to  (-.1,0.1); 
\draw (0,-.1)  to  (0,-.9);
\draw[blue,fill] (-2.5,2.5) circle [radius=1pt]
node[right] {};
\draw[blue,fill] (-.5,2.5) circle [radius=1pt]
node[right] {};
\draw[blue,fill] (-1.5,1.5) circle [radius=1pt]
node[left] {$\{c,d\}$};
\node[left] at (-0.5,0.5) {};
\draw (-1.4,1.6)  to  (-.6,2.4);
\draw (-1.6,1.6) to  (-2.4,2.4);

\end{tikzpicture}
    \caption{A factorization for the poset $\{r<s,r<t, q<s,q<t\}$.\label{fig:fact}
}

\end{figure}
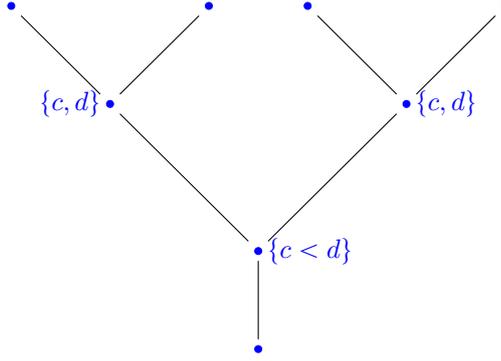

%Add definition
\begin{definition}
 Let $P$ be a $SP$-poset. The set $\hom_{Shuffle}(P,\chain[n])$ of (colimit-indexing) shuffle products of $P$ and $\chain[n]$ contain those posets $A$ satisfying:

\begin{itemize}
\item Every point in $A$ is labeled by a point in $P$ or a point in $\chain[n]$.
\item The number of minimum points of $A$ is the same as the number of minimum points in $P$.
\item The number of maximal points in $A$ is the same as the number of maximal points in $P$.

\item There is an isomorphism between maximal chains in $P$ and maximal chains in $A$: {For} any maximal chain $M$ in $P$, the corresponding chain in $A$ is a shuffle between the elements of $M$ and $\chain[n]$.
\end{itemize}   
\end{definition}

\begin{remark}
 When $P=\chain[k]$, the posets $\chain[k]$ and $\chain[n]$ have both one maximun, and one minimum. Then, the shuffles of those posets are the usual $(k,n)$-shuffles. 
\end{remark}
\begin{example}
The seven shuffles of $
\{a<c,a<d,b<c,b<d\}$ and the poset $\chain[1]$ are displayed in Figure~\ref{fig:7}.

\begin{center}
\begin{figure}

\begin{tikzpicture}[scale=0.17]
\tikzstyle{every node}+=[inner sep=0pt]
\draw [black] (5.1,-13.7) circle (3);
\draw [black] (5,-3.6) circle (3);
\draw [black] (13.9,-13.7) circle (3);
\draw [black] (13.9,-3.4) circle (3);
\draw [black] (5.1,-22.4) circle (3);
\draw [black] (5.1,-22.4) circle (2.4);
\draw [black] (13.9,-22.4) circle (3);
\draw [black] (13.9,-22.4) circle (2.4);
\draw [black] (60.3,-23.8) circle (3);
\draw [black] (60.3,-13.7) circle (3);
\draw [black] (70.9,-13.7) circle (3);
\draw [black] (70.9,-23.8) circle (3);
\draw [black] (60.3,-3.4) circle (3);
\draw [black] (60.3,-3.4) circle (2.4);
\draw [black] (70.9,-3.4) circle (3);
\draw [black] (70.9,-3.4) circle (2.4);
\draw [black] (24.8,-3.4) circle (3);
\draw [black] (44.9,-2.8) circle (3);
\draw [black] (24.8,-23.8) circle (3);
\draw [black] (44.9,-23.8) circle (3);
\draw [black] (44.9,-13.7) circle (3);
\draw [black] (44.9,-13.7) circle (2.4);
\draw [black] (38.7,-13.7) circle (3);
\draw [black] (38.7,-13.7) circle (2.4);
\draw [black] (30.9,-13.7) circle (3);
\draw [black] (30.9,-13.7) circle (2.4);
\draw [black] (24.8,-13.7) circle (3);
\draw [black] (24.8,-13.7) circle (2.4);
\draw [black] (7.05,-11.42) -- (11.95,-5.68);
\fill [black] (11.95,-5.68) -- (11.05,-5.96) -- (11.81,-6.61);
\draw [black] (13.9,-10.7) -- (13.9,-6.4);
\fill [black] (13.9,-6.4) -- (13.4,-7.2) -- (14.4,-7.2);
\draw [black] (11.92,-11.45) -- (6.98,-5.85);
\fill [black] (6.98,-5.85) -- (7.14,-6.78) -- (7.89,-6.12);
\draw [black] (5.1,-19.4) -- (5.1,-16.7);
\fill [black] (5.1,-16.7) -- (4.6,-17.5) -- (5.6,-17.5);
\draw [black] (13.9,-19.4) -- (13.9,-16.7);
\fill [black] (13.9,-16.7) -- (13.4,-17.5) -- (14.4,-17.5);
\draw [black] (60.3,-20.8) -- (60.3,-16.7);
\fill [black] (60.3,-16.7) -- (59.8,-17.5) -- (60.8,-17.5);
\draw [black] (62.47,-21.73) -- (68.73,-15.77);
\fill [black] (68.73,-15.77) -- (67.8,-15.96) -- (68.49,-16.68);
\draw [black] (70.9,-20.8) -- (70.9,-16.7);
\fill [black] (70.9,-16.7) -- (70.4,-17.5) -- (71.4,-17.5);
\draw [black] (68.73,-21.73) -- (62.47,-15.77);
\fill [black] (62.47,-15.77) -- (62.71,-16.68) -- (63.4,-15.96);
\draw [black] (60.3,-10.7) -- (60.3,-6.4);
\fill [black] (60.3,-6.4) -- (59.8,-7.2) -- (60.8,-7.2);
\draw [black] (70.9,-10.7) -- (70.9,-6.4);
\fill [black] (70.9,-6.4) -- (70.4,-7.2) -- (71.4,-7.2);
\draw [black] (27.23,-22.04) -- (36.27,-15.46);
\fill [black] (36.27,-15.46) -- (35.33,-15.53) -- (35.92,-16.34);
\draw [black] (24.8,-10.7) -- (24.8,-6.4);
\fill [black] (24.8,-6.4) -- (24.3,-7.2) -- (25.3,-7.2);
\draw [black] (44.9,-20.8) -- (44.9,-16.7);
\fill [black] (44.9,-16.7) -- (44.4,-17.5) -- (45.4,-17.5);
\draw [black] (44.9,-10.7) -- (44.9,-5.8);
\fill [black] (44.9,-5.8) -- (44.4,-6.6) -- (45.4,-6.6);
\draw [black] (42.47,-22.04) -- (33.33,-15.46);
\fill [black] (33.33,-15.46) -- (33.69,-16.33) -- (34.27,-15.52);
\draw [black] (29.37,-11.12) -- (26.33,-5.98);
\fill [black] (26.33,-5.98) -- (26.31,-6.92) -- (27.17,-6.41);
\draw [black] (5.07,-10.7) -- (5.03,-6.6);
\fill [black] (5.03,-6.6) -- (4.54,-7.4) -- (5.54,-7.39);
\draw [black] (24.8,-20.8) -- (24.8,-16.7);
\fill [black] (24.8,-16.7) -- (24.3,-17.5) -- (25.3,-17.5);
\draw [black] (40.18,-11.09) -- (43.42,-5.41);
\fill [black] (43.42,-5.41) -- (42.59,-5.86) -- (43.46,-6.35);
\end{tikzpicture}

\vspace{1cm}

\begin{tikzpicture}[scale=0.17]
\tikzstyle{every node}+=[inner sep=0pt]
\draw [black] (3.1,-11.4) circle (3);
\draw [black] (3.1,-28.2) circle (3);
\draw [black] (11.9,-11.4) circle (3);
\draw [black] (11.9,-28.2) circle (3);
\draw [black] (23,-11.4) circle (3);
\draw [black] (31.9,-11.4) circle (3);
\draw [black] (23,-28.2) circle (3);
\draw [black] (31.9,-28.2) circle (3);
\draw [black] (43.5,-11.4) circle (3);
\draw [black] (52.1,-11.4) circle (3);
\draw [black] (43.5,-28.2) circle (3);
\draw [black] (52.1,-28.2) circle (3);
\draw [black] (63.9,-11.4) circle (3);
\draw [black] (73.2,-11.4) circle (3);
\draw [black] (63.9,-28.2) circle (3);
\draw [black] (73.2,-28.2) circle (3);
\draw [black] (11.5,-39.9) circle (3);
\draw [black] (11.5,-39.9) circle (2.4);
\draw [black] (23,-39.9) circle (3);
\draw [black] (23,-39.9) circle (2.4);
\draw [black] (43.5,-3) circle (3);
\draw [black] (43.5,-3) circle (2.4);
\draw [black] (73.2,-3) circle (3);
\draw [black] (73.2,-3) circle (2.4);
\draw [black] (3.1,-19.8) circle (3);
\draw [black] (3.1,-19.8) circle (2.4);
\draw [black] (9.5,-19.8) circle (3);
\draw [black] (9.5,-19.8) circle (2.4);
\draw [black] (25.4,-19.8) circle (3);
\draw [black] (25.4,-19.8) circle (2.4);
\draw [black] (31.9,-19.8) circle (3);
\draw [black] (31.9,-19.8) circle (2.4);
\draw [black] (45.9,-19.8) circle (3);
\draw [black] (45.9,-19.8) circle (2.4);
\draw [black] (52.1,-19.8) circle (3);
\draw [black] (52.1,-19.8) circle (2.4);
\draw [black] (63.9,-19.8) circle (3);
\draw [black] (63.9,-19.8) circle (2.4);
\draw [black] (70.7,-19.8) circle (3);
\draw [black] (70.7,-19.8) circle (2.4);
\draw [black] (11.6,-36.9) -- (11.8,-31.2);
\fill [black] (11.8,-31.2) -- (11.27,-31.98) -- (12.27,-32.01);
\draw [black] (3.1,-25.2) -- (3.1,-22.8);
\fill [black] (3.1,-22.8) -- (2.6,-23.6) -- (3.6,-23.6);
\draw [black] (4.92,-25.81) -- (7.68,-22.19);
\fill [black] (7.68,-22.19) -- (6.8,-22.52) -- (7.59,-23.13);
\draw [black] (3.1,-16.8) -- (3.1,-14.4);
\fill [black] (3.1,-14.4) -- (2.6,-15.2) -- (3.6,-15.2);
\draw [black] (10.32,-16.92) -- (11.08,-14.28);
\fill [black] (11.08,-14.28) -- (10.38,-14.92) -- (11.34,-15.19);
\draw [black] (10.51,-25.54) -- (4.49,-14.06);
\fill [black] (4.49,-14.06) -- (4.42,-15) -- (5.31,-14.53);
\draw [black] (11.9,-25.2) -- (11.9,-14.4);
\fill [black] (11.9,-14.4) -- (11.4,-15.2) -- (12.4,-15.2);
\draw [black] (23,-36.9) -- (23,-31.2);
\fill [black] (23,-31.2) -- (22.5,-32) -- (23.5,-32);
\draw [black] (30.06,-25.83) -- (27.24,-22.17);
\fill [black] (27.24,-22.17) -- (27.33,-23.11) -- (28.12,-22.5);
\draw [black] (31.9,-25.2) -- (31.9,-22.8);
\fill [black] (31.9,-22.8) -- (31.4,-23.6) -- (32.4,-23.6);
\draw [black] (31.9,-16.8) -- (31.9,-14.4);
\fill [black] (31.9,-14.4) -- (31.4,-15.2) -- (32.4,-15.2);
\draw [black] (24.58,-16.92) -- (23.82,-14.28);
\fill [black] (23.82,-14.28) -- (23.56,-15.19) -- (24.52,-14.92);
\draw [black] (23,-25.2) -- (23,-14.4);
\fill [black] (23,-14.4) -- (22.5,-15.2) -- (23.5,-15.2);
\draw [black] (24.4,-25.55) -- (30.5,-14.05);
\fill [black] (30.5,-14.05) -- (29.68,-14.52) -- (30.56,-14.99);
\draw [black] (44.32,-25.32) -- (45.08,-22.68);
\fill [black] (45.08,-22.68) -- (44.38,-23.32) -- (45.34,-23.59);
\draw [black] (52.1,-25.2) -- (52.1,-22.8);
\fill [black] (52.1,-22.8) -- (51.6,-23.6) -- (52.6,-23.6);
\draw [black] (47.68,-17.39) -- (50.32,-13.81);
\fill [black] (50.32,-13.81) -- (49.44,-14.16) -- (50.25,-14.75);
\draw [black] (52.1,-16.8) -- (52.1,-14.4);
\fill [black] (52.1,-14.4) -- (51.6,-15.2) -- (52.6,-15.2);
\draw [black] (50.73,-25.53) -- (44.87,-14.07);
\fill [black] (44.87,-14.07) -- (44.79,-15.01) -- (45.68,-14.55);
\draw [black] (43.5,-25.2) -- (43.5,-14.4);
\fill [black] (43.5,-14.4) -- (43,-15.2) -- (44,-15.2);
\draw [black] (43.5,-8.4) -- (43.5,-6);
\fill [black] (43.5,-6) -- (43,-6.8) -- (44,-6.8);
\draw [black] (63.9,-25.2) -- (63.9,-22.8);
\fill [black] (63.9,-22.8) -- (63.4,-23.6) -- (64.4,-23.6);
\draw [black] (72.34,-25.32) -- (71.56,-22.68);
\fill [black] (71.56,-22.68) -- (71.3,-23.58) -- (72.26,-23.3);
\draw [black] (63.9,-16.8) -- (63.9,-14.4);
\fill [black] (63.9,-14.4) -- (63.4,-15.2) -- (64.4,-15.2);
\draw [black] (68.81,-17.47) -- (65.79,-13.73);
\fill [black] (65.79,-13.73) -- (65.9,-14.67) -- (66.68,-14.04);
\draw [black] (65.35,-25.58) -- (71.75,-14.02);
\fill [black] (71.75,-14.02) -- (70.92,-14.48) -- (71.8,-14.97);
\draw [black] (73.2,-25.2) -- (73.2,-14.4);
\fill [black] (73.2,-14.4) -- (72.7,-15.2) -- (73.7,-15.2);
\draw [black] (73.2,-8.4) -- (73.2,-6);
\fill [black] (73.2,-6) -- (72.7,-6.8) -- (73.7,-6.8);
\end{tikzpicture}

\begin{caption}\, Shuffles of $\{a<c,a<d,b<c,b<d\}$ and $\chain[1]$, the points labeled by the second poset are drawn with double circles.
\label{fig:7}
    \end{caption}
\end{figure}
\end{center}

\end{example}

%Add definition
\begin{definition}
Let $Q$ be a ${SP}$-poset, define the (colimit-indexing) shuffle series of $Q$ by $$\mathcal{SH}(Q)=\sum_{n=0}^\infty \#\hom_{Shuffle}(Q,\chain[n])x^n.$$    
\end{definition}
  Here, we assume that there is one shuffle between a poset and the empty set, the poset itself. For a chain, we denote $\mathcal{SH}(k):=\mathcal{SH}(\chain[k])$.

Denote by $\multiset{n}{k}$ the multiset coefficient $\multiset{n}{k}={n+k-1\choose k}$. Let $m\in\mathbb{N}$, then we compute 
\begin{align*}
    \mathcal{SH}(m)=&\sum_{n=0}^\infty \#\hom_{Shuffle}(\chain[m],\chain[n])x^n\\ =&\sum_{n=0}^\infty \multiset{n+1}{m}x^n\\ =&\frac{1}{(1-x)^{m+1}}.
\end{align*}
%                -----------5

\subsection{Combinatorial definition of the action of the operad $\mathcal{SP}$ on series}
%                                 1-----

\label{sec:exp}

%The generating series of $P$ is denoted by $\zetaf[P]=\sum_{n=1}^\infty \somega[n]{P}x^n$. 
The factorizations of posets in terms of the lexicographic sum (see Definition~\ref{LS}) are studied in~\cite[Chapter 7.2]{osets}. A poset is indecomposable if the only factorizations are trivial $P=\widetilde{\chain[1]}(P)$ or $P=\widetilde{P}(\chain[1],\cdots,\chain[1])$. For example, the poset $\{w,x,y,z|x<y>w<z\}$ and the poset $\{x,y|x<y\}$ are indecomposable.
In general, the factorization of a poset is not unique.

We aim to study sets in which posets induce operations. To formalize this idea we use the language of operads.
We follow~\cite{OpTre, SandDHT} and we recommend~\cite{whatare,whatis} for an introduction.

%Add definition
\begin{definition}
 The operad of finite posets $\mathcal{FP}$ contains as $n$-ary operations posets with $n$-points, and the operadic composition is the lexicographic sum. 
\end{definition}

Elements of an operad are written with a tilde, to distinguish the poset $P=(\{x_1,\cdots,x_n\},\leq_P)$ and the $n$-ary operation $\widetilde{P}$.

The operadic composition is compatible with the action of poset automorphisms (order preserving isomorphisms). 

 To control the order of composition of the posets, we will use trees as a visual aid. 
 \begin{definition}
 We consider the $n-$ary operations of the operad of finite posets to be the set of trees (finite acyclic graphs with a chosen external edge called the root) with $n$ leaves in which vertices with $k$ children are labeled by an indecomposable poset with $k$ points. If a vertex is labeled by a poset $P$, then the outgoing edges of the vertex are labeled by the points of the poset $P$.  
  The composition is given by the grafting of trees, which identifies the root of a tree with the leaf of a second tree.
 \end{definition}
\begin{remark}
From now on, we associate the operation $\widetilde{P}$ without inputs, where $P$ is an indecomposable poset with $|P|=n$ points, to the $n$-corolla labeled by $P$. Next, for $Q_1,\cdots, Q_n$ posets, we consider them as operations, with their corresponding corollas $\widetilde{Q_1}, \cdots,\widetilde{Q_n}$. Then we define the tree $\widetilde{P}(\widetilde{Q_1}, \cdots, \widetilde{Q_n})$ as the grafting of the corollas $\widetilde{Q_1}, \cdots,\widetilde{Q_n}$ along the $n$-corolla $\widetilde{P}$, where the $i$ corolla $\widetilde{Q_i}$ is grafted on the $i$ outgoing edge of $\widetilde{P}$.
%a tree formed from an  $n$-corolla labeled by $P$, in which the children vertices of the corolla are labeled by $Q_1, \cdots, Q_n$, and each children vertex has as many leaves as points in the corresponding $|Q_i|$.
Similarly for more complex operations. In this paper we will always assume this assignment. %which forces the grafting operation to coincide with the lexicographic sum of posets.    
\end{remark}

Our approach differs from~\cite{operadskoszul}, in their construction \textit{each} poset defines an associative operad.

In Figure~\ref{Fig:ternary operation} we see an example of a four-ary operation that is the result of the grafting of two binary operations $(\widetilde{\{c,d\}},\widetilde{\{c<d\}})$ on the binary operation $\widetilde{\{c<d\}}$. 
  \begin{figure}[htb]\centering
\begin{tikzpicture}[xscale=1.3,yscale=1.3]
 
\draw[blue,fill] (1.5,1.5) circle [radius=1pt]
node[right] {$\{c<d\}$};
\draw[white,fill] (0.5,2.5) circle [radius=1pt]
node[right] {};
\draw[white,fill] (2.5,2.5) circle [radius=1pt]
node[right] {};
\draw (.1,.1)  to  (1.4,1.4);
\draw (1.4,1.6)  to  (.6,2.4);
\draw (1.6,1.6) to  (2.4,2.4);

\draw[blue,fill] (0,0) circle [radius=1pt]
node[right] {$\{c<d\}$};
\draw[white,fill] (0,-1) circle [radius=1pt];
\node[left] at (0.1,-0.5) {};

\draw (-1.4,1.4)  to  (-.1,0.1); 
\draw (0,-.1)  to  (0,-.9);
\draw[white,fill] (-2.5,2.5) circle [radius=1pt]
node[right] {};
\draw[white] (-.5,2.5) circle [radius=1pt]
node[right] {};
\draw[blue,fill] (-1.5,1.5) circle [radius=1pt]
node[left] {$\{c,d\}$};
\node[left] at (-0.5,0.5) {};
\draw (-1.4,1.6)  to  (-.6,2.4);
\draw (-1.6,1.6) to  (-2.4,2.4);

\end{tikzpicture}
\caption{The four-ary operation $\widetilde{\{c<d\}}(\widetilde{\{c,d\}},\widetilde{\{c<d\}})$ evaluated in four points $(p,q,r,s)$ returns the poset $\{p,q,r,s|p<r,q<r,r<s\}.$\label{Fig:ternary operation}
}
\end{figure}

\begin{definition}
    To any set $X$  we associate the operad $End_X$ with $n$-ary operations $\hom_{Sets}(X^n,X)$. The set $X$ is an algebra over an operad $O$ if there exists a morphism of operads (a morphims compatible with the composition) $O\rightarrow End_X$.
    \end{definition}
    An abstract $n$-ary operation $\widetilde{F}\in \mathcal{O}(n)$ induces a function ${F}:X^n\rightarrow X$.

 \begin{remark}
     We proceed to describe explicitly the set of posets as an algebra over the operad $\mathcal{FP}$. Given a tree $\widetilde{P}$ where $P$ is a poset with $n$ points, and $Q_1,\cdots, Q_n$ posets, we visualize $\widetilde{P}(Q_1,\cdots, Q_n)$ as placing each poset $Q_i$ on its  corresponding leaf $i$, and thinking of them as drops of water. From top to bottom, as the drops of water reach an internal vertex, we apply the corresponding lexicographic sum and the new drop of water continues descending. 
     That is, if the vertex is labeled with the ( indecomposable ) poset $P_k$, with $|P_k|=l$ points, and the posets $T_1,\cdots,T_l$ correspond to the outgoing edges of the vertex $P_k$ then the output drop of water is the poset  $P_k(T_1,\cdots,T_l)$.
     Going from top to bottom, we reach the root, which should correspond to $P(Q_1,\cdots,Q_n)$.
 \end{remark}

Technically speaking, when we want to compute the action of the operad $\mathcal{FP}$ (or $\mathcal{SP}$) on a set, we need to choose a labeling for all the outgoing edges of vertices  ( including a labeling on all the leaves of the corresponding trees) which we use to indicate where is each input assigned.

For every finite poset, Stanley~\cite{beginning} defined  $\somega[n]{P}$ to be the number of strict poset maps between the poset $P$ and the $n$ chain. Note that if the poset $P$ is not empty, then there is no map into the empty set, then we declare $\somega[0]{P}=0$. 
\begin{definition}    
The strict order series are the generating functions of Stanley's strict order polynomials $\zetaf[P]=\sum_{n=1}^\infty \somega[n]{P}x^n$. \label{series:strict}
\end{definition}

We proceed to define the operad of series-parallel posets as a sub operad of the operad of finite posets.  While the structure of algebra over the operad of finite posets of strict order series is well understood (see \cite{APV}), we only undestand shuffle series as an algebra over the operad of series-parallel posets. For this reason we restrict most of our results to the operad of series parallel poset. In the second paragraph of Section~\ref{Sect6} we explain with more detail the obstructions encountered when defining the action of the operad of finite posets on shuffle series.

 \begin{definition}     
The operad of series-parallel posets $\mathcal{SP}$ is the suboperad of $\mathcal{FP}$ which contains as $n$-ary operations the set of binary trees with $n$ leaves in which vertices are labeled with $\widetilde{\{c,d\}}$ or  $\widetilde{\{c<d\}}$.  \end{definition}

\begin{remark}\label{rm:fact}
A factorization  of a ${SP}$-poset $P$, $|P|=n$, consist on finding a tree $\widetilde{Q}$ in $\mathcal{SP}(n)$ so that $\widetilde{Q}(\chain[1],\cdots,\chain[1])=P$.
\end{remark}

 For an ${SP}$-poset $P$, and any chain $\chain[n]$, assume that you can compute a set $\hom_C(P,\chain[n])$. Here we use $C$ to denote an operation. For example, we can consider $\hom_{Shuffle}(P,\chain[n]) $ the set of shuffles among the posets $P$ and $ \chain[n]$, or $\hom_{Poset}(P,\chain[n])$ the set of strict poset maps from $P$ to $\chain[n]$. 

We are interested in the cases in which the following properties hold:

\begin{itemize}
\item[$C_1$:] For every ${SP}$-poset $P$, the function $n\mapsto \#\hom_C(P,\chain[n])$ is a polynomial on $n$.

\item[$C_1$:]We can write $\#\hom_C\left(\widetilde{\{c,d\}}(\chain[j], \chain[k]),\chain[n]\right)=\sum_{i\in I} a_i \#\hom_C(\chain[i],\chain[n]), a_i\in\mathbb{Z}$, with $I$ a finite set of indices.
\end{itemize}

For every ${SP}$-poset $P$ consider the generating function of $\#\hom_C$ with second variable running over the chains:  $Series_C(P)=\sum_{n=0}^\infty \#\hom_C(P,\chain[n]) x^n$.

Posets $P$ and $Q$ are Doppelg\"angers~\cite{Doppelgangers2,Doppelgangers,  doppelgangers3} if $\sum_{n=0}^\infty \#\hom_{Poset}(P,\chain[n]) x^n$ coincides with $\sum_{n=0}^\infty \#\hom_{Poset}(Q,\chain[n]) x^n$. For example, the order series of $\{w<x,w<y,w<z\}$ and $\{w<x, y<z\}$ coincide.
Then, the function transforming a poset to a series is not necessarily injective. We remember which poset generated every series by specifying the poset on the notation $\mathcal{SH}(P), \zetaf[P]$, etc. Technically speaking our objects are pairs consisting of the series $\sum_{n=0}^\infty \#\hom_C(P,\chain[n]) x^n$ and the poset $P$. With this extra information, there is an isomorphism between the set of %${SP}$-
posets and the set of their associated series, the inverse map is $(\sum_{n=0}^\infty \#\hom_C(P,\chain[n]) x^n, P)\rightarrow P$.

The condition $C_1$ implies that the functions $n\mapsto\#\hom_C(P,n)$ are integer-valued polynomials on the variable $n$, and as polynomials, they are the sum of binomial coefficients via the Gregory-Newton decomposition~\cite{IVP, AMS}. The condition $C_2$ guarantees that the series associated with the disjoint union of two chains is a linear combination of the generating series of chains $Series_{C}(\widetilde{\{c,d\}}(P, Q))=\sum_{i\in I} a_iSeries_C(\chain[i])$,  $I$ a finite set of indices.

%add definition
Let $C-series$ denote the set $\{Series_C(P)|P\in \mathcal{SP}-\hbox{poset}\}$. 

\begin{definition}\label{action:FP}
Given $P, P_1,\cdots,P_k, \in \mathcal{SP}$ poset, $k=|P|$, define the action of the operad $\mathcal{SP}$ on $C-series$ by: \[\widetilde{P}\left(Series_C(P_1),\cdots,Series_C(P_k)\right)=Series_C\left(\widetilde{P}(P_1,\cdots,P_k)\right).\]
\end{definition}

While every factorization of a poset with $n$ leaves evaluated on $n$ copies of the one chain returns the original poset, by definition of factorization, we still need to show that on other poset algebras the action of the operad is independent of the choice of the factorization.
%why do we need this? Assume that $C$ satisfies conditions $C_1$ and $C_2$.
 \begin{lemma}\label{lemma:WD}  Let $\gamma(\widetilde{\{c<d\}})$ and $\gamma(\widetilde{\{c,d\}})$ denote the action of $\widetilde{\{c<d\}}$ and $\widetilde{\{c,d\}}$ on $C$-Series respectively. If $\gamma(\widetilde{\{c<d\}})$ is associative and $\gamma(\widetilde{\{c,d\}})$ is associative and commutative, then the action of the operad $\mathcal{SP}$ on $C$-Series is well defined.
\end{lemma}
\begin{proof}  
We follow the point of view of combinatorial reconfiguration~\cite{CRec}. First{ly}, we describe the space of factorizations of a poset. Secondly, we find a family of transformations $\{T_i\}$ so that one can link any two factorizations of the same poset by applying some sequence of those transformations. Then, we argue that the action of the operad is not changed when we apply any element of $\{T_i\}$ to a factorization.

What is the set of factorizations in our case? Given a pair, $x,y$ with $y$ a successor of $x$, every factorization must have $\widetilde{\{c<d\}}$ applied to (antecessor($y$), successor($x$)) according to Lemma~\ref{Lemma:N}. The reasons for which factorizations {may} differ are the way the sets antecessor($y$) and 
successor($x$) are built out of the associative operation $\widetilde{\{c,d\}}$, the factorizations can differ by the order of application of the associative operation $\widetilde{\{c<d\}}$ to build chains, and the previous two cases applied not to points but to sets of posets.

What sequence of transformations defined on factorizations relates any two factorizations of a fixed poset? Consider the transformation $\sigma_{1,2}(v)$ which switch the inputs of a vertex $v$ labeled by $\widetilde{\{c,d\}}$, and the transformation $a(v,v)$ which transforms $v(\,,v(\,,\,))$ into $v(v(\,,\,),\,)$ when $v$ is a vertex labeled by a binary poset. Then, according to the description of the factorizations, any two factorizations of the same poset can be linked by a series of applications of $a$ and $\sigma_{1,2}$.

Since $\gamma(\widetilde{\{c,d\}})$ is commutative, two trees that differ by the transformation $\sigma_{1,2}$ will compute the same order series. 
The other transformation $a$ encodes associativity and since both operations $\gamma(\widetilde{\{c,d\}})$ and $\gamma(\widetilde{\{c<d\}})$ are associative, then two trees that differ by the transformation $a$ compute the same series.
\end{proof}

Following the proof of Theorem 2.4 of \cite{APV}, we obtain.  %~\cite[Theorem 2.4]{posetzeta}
\begin{lemma}\label{Lemma:act}Under the hypothesis of Lemma~\ref{lemma:WD},
 there is an operadic morphism among algebras over the operad of series parallel posets:
$$Series_C: P\in SP- Posets\rightarrow C-Series.$$
\label{Lemma:tool}
\end{lemma}\begin{proof}
Since our objects are pairs, a series and the poset that generated the series, there is a set isomorphism between posets and the objects under study. The isomorphism of sets induces an isomorphism of algebras over the operad $\mathcal{SP}$. 
\end{proof}

There are two problems with our current approach, firstly, to define the action we need to keep track of the posets that generate each series, see the definition of  $Series_C(P)$.  Secondly, we still do not know the explicit formulas to compute the action. The first problem is unavoidable because of Doppelg\"anger posets. In the next section, we will solve the second problem.

\begin{lemma}\label{Lemma:aut}
Let $\phi:Posets\rightarrow Posets$ be an automorphism of posets, then for $n>1$ we have $$\#\hom_{Poset}(P,\chain[n])=\#\hom_{Poset}(\phi(P),\chain[n]).$$ \label{Poset:aut}
\end{lemma}
\begin{proof}
The assignment $f\mapsto f\circ\phi^{-1}$ from $\hom_{Poset}(P,\chain[n])$ to $\hom_{Poset}(\phi(P), \chain[n])$ and $g\mapsto g\circ\phi$ from $\hom_{Poset}(\phi(P),\chain[n])$ to $\hom_{Poset}(P, \chain[n])$ are injective and inverse to each other.
\end{proof}

By the Lemma~\ref{Lemma:aut}, automorphisms of posets act on order series trivially. Similarly, the number of shuffles between a poset $P$ and a chain does not changes when we replace the poset by an isomorphic one.

\subsection{The $\mathcal{SP}$ algebra generated by a point}

 We think of series-parallel posets as an algebra over the operad $\mathcal{SP}$ generated by one point, see Remark~\ref{rm:fact}.

Let $A$ be an algebra over the operad $\mathcal{SP}$. Assume that there is an element $point\in A$ so that for every $\widetilde{P}\in \mathcal{SP}$, the evaluation $\widetilde{P}(point,\cdots,point)$ is well defined, here the vector has as many entries as the number of leaves on the factorization of $P$.
%Add definition
\begin{definition}
    The $\mathcal{SP}$ algebra generated by $point$ contains all possible binary trees whose vertices are labeled by posets with two points and every leaf has the object $point$ \[\sqcup_n \{\widetilde{P}(point,\cdots,point)| \widetilde{P}\in \mathcal{SP}(n), \hbox{ and the vector has exactly }n \hbox{ coordinates}\}.\]  
\end{definition}

Since the algebra is obtained by evaluating trees at $(point,\cdots, point)$, the action of $\widetilde{P}$ on elements $a_1,\cdots, a_k$ can be described by choosing a factorization ( the factorization may not be unique) of every $a_i=\widetilde{P}_i(point,\cdots,point)$ and grafting the corresponding trees: 

\begin{eqnarray*}
\widetilde{P}(a_1,\cdots,a_k)&=&\widetilde{P}(\widetilde{P}_1(point,\cdots,point),\cdots,\widetilde{P}_k(point,\cdots,point))\\
&=&\widetilde{P}(\widetilde{P_1},\cdots,\widetilde{P_k})(point,\cdots,point),
\end{eqnarray*}
here the last vector $(point,\cdots, point)$ has as many entries as the sum of the numbers of entries for the fixed factorization of each $a_i$. We have to choose a factorization for $\widetilde{P}$ as well as factorizations for every element $a_i$. 
 
A operadic morphism $\phi$ between the $\mathcal{SP}$-algebra generated by  $point$ and another $\mathcal{SP}$-algebra $F$ is determined by the image of $point$. Effectively, an arbitrary element on the image of $\phi$ can be written as 
$\widetilde{P}(\phi(point),\cdots,\phi(point))$. The grafting of trees preserves this representation, for example:
$\widetilde{P}(\widetilde{Q}(point, point),\widetilde{R}(point))$ which is equivalent to the tree $\widetilde{P}(\widetilde{Q},\widetilde{R})$ evaluated on $(point, point, point)$ is map to the tree $\widetilde{P(Q,R)}$ evaluated on 
$(\phi(point),\phi(point),\phi(point))$, which is the same as $\widetilde{P}(\widetilde{Q}(\phi(point), \phi(point)),\widetilde{R}(\phi(point)))$. 

\subsection{The Shuffle series as $SP$ algebra generated by $\frac{1}{(1-x)^2}$.}
We now apply the theory to shuffle series.

By the definition of shuffle, a shuffle of the union of $P$ and $Q$ with a chain is a pair consisting of a shuffle of $P$ with the chain and a shuffle of $Q$ with the chain, then 
\begin{equation}\label{eq:sh}
\#\hom_{Shuffle}(\widetilde{\{c,d\}}(\chain[n],\chain[m]),\chain[l])=\#\hom_{Shuffle}(\chain[n],\chain[l])\#\hom_{Shuffle}(\chain[m],\chain[l]),\end{equation}
which is compatible with the computations of~\cite{ShuffleTree}, alternatively, we would like to use~\cite[Corollary 5.29]{CatBook}. At the level of generating functions, we require $\widetilde{\{c,d\}}(\mathcal{SH}(P), \mathcal{SH}(Q))$ to be the Hadamard product of the corresponding series.

\begin{lemma}\label{lemmaU}
If $n\geq m$, and $l\in\mathbb{N}$, then
\begin{equation}
    \#\hom_{shuffle}\left(\widetilde{\{c,d\}}(\chain[n], \chain[m]),\chain[l]\right)\nonumber
    \end{equation}
\begin{equation}
    =\sum_{r=0}^m(-1)^{n+m-(r+n)} {n+r\choose m}{m\choose r}\#\hom_{shuffle}(\chain[n+r],\chain[l])\nonumber\label{Eqn:c2}
\end{equation}
\end{lemma}
\begin{proof}

We computed $\#\hom_{shuffle}(\chain[n],\chain[l])={l+n\choose n}=\multiset{l+1}{n}$. The identity \cite[6.44]{lcomb}, for $n\geq m$:
\begin{eqnarray}\label{Eq:id}
  {x\choose n}{x\choose m}&=&\sum_{r=0}^m {n+r\choose m}{m\choose r} {x\choose n+r},  
\end{eqnarray}
can be seen as an identity of polynomials, and we replace $x\rightarrow -x$ to obtain
\begin{eqnarray}\label{Eq:id2}
  (-1)^{n+m}\multiset{x}{n}\multiset{x}{m}&=&\sum_{r=0}^m(-1)^{n+r} {n+r\choose m}{m\choose r} \multiset{x} {n+r},  
\end{eqnarray}

by using $(-1)^b{-a\choose b}=\multiset{a}{b}$. If we replace $x$ by $l+1$ we obtain 
\begin{eqnarray*}\lefteqn{
    \#\hom_{shuffle}(\widetilde{\{c,d\}}(\chain[n], \chain[m]),\chain[l])}\\&=&\sum_{r=0}^m(-1)^{m+r} {n+r\choose m}{m\choose r}\#\hom_{shuffle}(\chain[n+r],\chain[l]).
\end{eqnarray*}

\end{proof}

In particular, at the level of generating functions, we obtain from Equation~\eqref{Eqn:c2} the explicit formula for the shuffle of chains, if $n\geq m$: 
\begin{equation}\label{eqn:U}
  \widetilde{\{c,d\}}(\mathcal{SH}(n),\mathcal{SH}(m))= \sum_{r=0}^{m}(-1)^{n+m-(n+r)}{n+r\choose m}{m\choose r}\mathcal{SH}(n+r).  
\end{equation}

\begin{lemma} The action of $\widetilde{\{c<d\}}$ on chains is given by
$$\widetilde{\{c<d\}}(\mathcal{SH}(k),\mathcal{SH}(l))=\mathcal{SH}(k)(1-x)\mathcal{SH}(l)=\mathcal{SH}(k+l),$$\label{lemma:ch}
where $k,l\in\mathbb{N}$. 
\end{lemma}

\begin{proof}
First note that the action of $\widetilde{\{c<d\}}$ on shuffle series of chains is $\widetilde{\{c<d\}}(\mathcal{SH}(k),\mathcal{SH}(l)) = \mathcal{SH}(k+l)$. 

Chains of different sizes have assigned functions of different exponent, then the shuffle series assignment: $P\rightarrow \mathcal{SH}(P)$ is bijective on chains, and we can forget the poset label. The operation $\widetilde{\{c<d\}}$ acts on shuffles of chains by $$\widetilde{\{c<d\}}(\frac{1}{(1-x)^{k+1}},\frac{1}{(1-x)^{l+1}})=\frac{1}{(1-x)^{k+l+1}}.$$   The only operation with that behaviour is the deformed product $f,g\mapsto f(1-x)g$.
\end{proof}

As a corollary of Theorem~\ref{main:thrm}, one can follow the explicit construction~\cite[Proposition 2.4, Proposition 2.5, Equation (7)]{posets} which provides an alternative explanation of the term $1-x$ in  Lemma~\ref{lemma:ch}. The alternative explanation uses splitting sequences of sets. Compare with the proof of ~\cite[Proposition 2.13]{posets} and the comments after Lemma~\ref{Lemma:lsh}.

Since the Hadamard product is associative and commutative, and the deformed Cauchy product of series in Lemma~\ref{lemma:ch} is commutative, the action of $\mathcal{SP}$ is well defined and we can apply the results of the previous section.

  We denote by $\mathcal{SH}(\varnothing)=\frac{1}{1-x}$, this series is the unit under the action of $\widetilde{\{c,d\}}$ and $\widetilde{\{c<d\}}$. See Remark~\ref{Remark:maya} for an explanation of the origin of this notation.

Once we know the action of the posets generators under lexicographic sum on shuffles of chains, we can compute $\widetilde{P}(\mathcal{SH}(1),\cdots,\mathcal{SH}(1))$. The following theorem is of combinatorial type, it explains that by computing the series  $\widetilde{P}(\mathcal{SH}(1),\cdots,\mathcal{SH}(1))$ we solve the problem of counting shuffles between $P$ and chains.

\begin{theorem}The $\mathcal{SP}$-algebra generated by $\mathcal{SH}(1)=\frac{1}{(1-x)^2}$ with the action of $\widetilde{\{c,d\}}(f,g)$ given by the Hadamard product of $f$ and $g$, and the action of $\widetilde{\{c<d\}}(f,g)=f(1-x)g$, is the $\mathcal{SP}$-algebra of shuffle series.\label{Thm:equality}
\end{theorem}\begin{proof}
Let $P$ be a $SP$-poset. We claim that when we compute the series $\widetilde{P}(\mathcal{SH}(1),\cdots,\mathcal{SH}(1))$, starting from the leaves to the root, at every stage of the computation we obtain a linear combination of chains.

Firstly, consider a vertex whose two incoming edges are leaves, if we find the label $\widetilde{\{c,d\}}$ acting as the Hadamard product of the shuffles series of two chains, from Equation~\eqref{eqn:U} 
the result is a linear combination of shuffle series of chains,  $\widetilde{\{c,d\}}(\mathcal{SH}(1),\mathcal{SH}(1))=  2\mathcal{SH}(2)-\mathcal{SH}(1)$. On the other hand, if we find $\widetilde{\{c<d\}}$ with inputs $(\mathcal{SH}(1),\mathcal{SH}(1))$, we obtain a linear combination of chains as output by Lemma~\ref{lemma:ch}. 

Now, iteratively, if the inputs of a vertex are two linear combinations of chains, then we use the distributivity of the Hadamard product with respect to addition or the distributivity of the multiplication by $(1-x)$  with respect to addition, to simplify $\widetilde{\{c,d\}}\big( \sum_i a_i\mathcal{SH}(i),\sum_j b_j\mathcal{SH}(j) \big)$, into $$\sum_i\sum_j a_ib_j \widetilde{\{c,d\}}\big(\mathcal{SH}(i),\mathcal{SH}(j) \big),$$
similarly, $$ \widetilde{\{c<d\}}\big( \sum_i a_i\mathcal{SH}(i),\sum_j b_j\mathcal{SH}(j) \big)=\sum_i\sum_j a_ib_j \widetilde{\{c<d\}}\big(\mathcal{SH}(i),\mathcal{SH}(j) \big).$$

We repeat the previous process until we reach the root.

The series $\mathcal{SH}(1)$ counts shuffles between $\chain[1]$ and other chains by definition. Let $v$ be a vertex on the factorization of $P$ and let $P_v$ be the poset whose factorization is the subtree of the factorization of $P$ with root $v$. From Lemma~\ref{lemmaU} and Lemma~\ref{lemma:ch} we guarantee that at vertex $v$ the series obtained by evaluating the factorization of $P_v$ on $(\mathcal{SH}(1),\cdots,\mathcal{SH}(1))$  (with as many entries as needed) is counting shuffles between the poset $P_v$ and chains. Then, when we reach the root, the final series counts the shuffles of $P$ with chains, which is the definition of $\mathcal{SH}(P)$.
\end{proof}
   
\begin{corollary}
    For every $SP$-poset $P$, there is a polynomial that evaluated on $n$ has the value $\#\hom_{shuffle}(P,\chain[n])$. The polynomials $\#\hom_{shuffle}(P,x)$ satisfy condition $C_2$.
\end{corollary}
\begin{proof} 
The condition $C_1$ states that the function $n\rightarrow \#\hom_{Shuffle}(Q,\chain[n]), n\in\mathbb{N},$ is a polynomial on the variable $n$. The condition $C_1$ is proved for the poset of vertices of a tree in~\cite[Proposition 3.5]{ShuffleTree} by using Lemma~\ref{Lemma:reconciliation}. 
The proof is iterative, starting from corollas, and then lexicographic sums of corollas on a corolla, etc.

To prove condition $C_1$ for series parallel posets, note that to any chain $\chain[m]$ we compute $\#\hom_{Shuffle}(m,n)=\multiset{n+1}{m}$, where $\multiset{x}{b}={x+b-1\choose b}=\frac{(x+b-1)(x+b-2)\cdots(x)}{b!}$.

By Theorem~\ref{Thm:equality} we obtain $\mathcal{SH}(P)=\sum c_i\mathcal{SH}(i)$. Then, we obtain the explicit formula 
\begin{equation}
    \#\hom_{shuffle}(P,x)=\sum_i c_i\multiset{x+1}{i}.\label{eq:ci}
\end{equation}

Condition $C_2$ follows from Lemma~\ref{lemmaU}.
\end{proof}

Section~\ref{Sec:Third} describes everything we know about the 
coefficients $c_i$ in Equation~\eqref{eq:ci}.

%                                  ----------------------6

\subsection{Relation with tree shuffles}\label{TS}
Our main motivation to write this paper is to count the (graph) shuffles between trees and linear trees, an open problem described in the paper~\cite[Page 64]{ShuffleTree}. In this section, we explain the relationship between posets and the (graph) trees of the aforementioned paper. 
To prevent confusion, we use the letters $P,Q$ for posets and $S,T$ for trees. We review the main definitions from~\cite{ShuffleTree} and refer the reader to the original paper for details.

A linear tree is a pair $(V,E)$ of vert{ices} and edges where the set of edges is $0,1,\cdots, n$, every vertex has one incoming edge and one outgoing edge, and there is one special edge called the {$0$ root} and one special edge called the {$n$ leaf}.  Besides the root and the leaf, any other edge connects two vertices.
Let $(V=\{0,\cdots,n\},E)$ and $(V_2=\{0,\cdots,m\},E_2)$ be two linear trees. Consider the rectangle with vertices $(a,b)$ with $a\in E$ and $b\in E_2$.
A classical shuffle of linear trees $(V,E)$ and $(V_2,E_2)$ 
is given by a path from $(0,0)$ to $(n,m)$ using only movements of the form $(r,s)\mapsto(r+1,s)$ or $(r,s)\mapsto(r,s+1)$.   

A tree is a finite graph without cycles, whose external edges are connected to one vertex only, and one of those external edges is called the root while the other external edges are called leaves.
% Add definition
\begin{definition}
Let $S$ and $T$ be two trees. A shuffle of $S$ and $T$ is a tree $A$ satisfying:
\begin{itemize}
\item[(1)] The edges of $A$ are labeled by pairs $(s,t)$ where $s$ and $t$ are edges of $S$ and $T$, respectively.
\item[(2)] The root of $A$ is labelled by the pair $(r_s,r_t)$ of root edges of $S$ and $T$. 
\item[(3)] The set of labels of the leaves of $A$ is equal to the cartesian product of the leaves of $S$ and those of $T$.
\item[(4)] For any two leaves $s$ {in} $S$ and $t$ {in} $T$, the branch from the leaf $(s,t)$ {in} $A$ down to its root, is a classical shuffle, of the two branches from $s$ down to the root in $S$ and from $T$ down to the root in $T$.
\end{itemize}
\end{definition}

Let $T$ be a linear tree and $S$ be an arbitrary tree. In this section {we} will denote the set of shuffles of $S,T$ by $Sh_{Graph}(S,T)$, and the set of shuffles of a poset $P$ with a chain $\chain[n]$ by $Sh_{Poset}(P,n)$. From \cite[Corollary 4.5]{ShuffleTree} ${Sh}_{Graph}(S,T)={Sh}_{Graph}(S_{red},T_{red})$, where the reduction of the tree $S$ is the tree $S_{red}$ obtained by pruning away all leaves, and putting exactly one leaf on each top vertex. 

To any reduced tree $T$, we define $\phi(T)$ the poset obtained by removing the external edges (the root and the leaves), and considering the resulting graph as the Hasse diagram of a poset $U(T)$, we use the order on $U(T)$ in which the root vertex is the minimum.
Note that this assignment $T\mapsto U(T)$ is an isomorphism.

The unit tree has an edge and no vertices, it is denoted by $\eta$, and we declare $U(\eta):=\varnothing$.

\begin{lemma}If $S$ is a reduced tree, and $T$ is a linear tree, then:
\[\#{Sh}_{Graph}(S,T)=\#{Sh}_{Poset}(U(S),U(T)).\]\label{Lemma:reconciliation}
\end{lemma}
\begin{proof}
Let $P\in {Sh}_{Graph}(S,T)$ and $\phi$ the isomorphism from reduced trees to a subset of posets. 

Now $P$ is also a reduced tree. And since $\phi$ is an isomorphism, it is enough to show that $\phi(P)$ is a shuffle of posets. Note that $\phi$ transforms the $4^{th}$ condition on the definition of shuffle of graphs into the $4^{th}$ condition on the definition of shuffles of trees, and preserves the other definitions.
\end{proof}

\begin{example}
Consider the tree $R$ in \cite[Example 3.6]{ShuffleTree}, using Equation~\ref{eqn:U} and Lemma~\ref{lemma:ch} we obtain:

\begin{eqnarray*}
    \mathcal{SH}\Big(\widetilde{\{c<d\}}\big(\left(\widetilde{\{c,d\}}(\chain[2],\chain[2])\right),\chain[1]\big)\Big)&=&\widetilde{\{c<d\}}\big(\left(\widetilde{\{c,d\}}(\mathcal{SH}(2),\mathcal{SH}(2))\right),\mathcal{SH}(1)\big)\\
    &=&\widetilde{\{c<d\}}\big(\mathcal{SH}(2)-6\mathcal{SH}(3)+6\mathcal{SH}(4),\mathcal{SH}(1)\big)\\
    &=&\mathcal{SH}(3)-6\mathcal{SH}(4)+6\mathcal{SH}(5)\\
    &=&\sum_{n=1}^\infty {n+3\choose 3}-6{n+4\choose 4}+6{n+5\choose 5} x^n
\end{eqnarray*}
Mathematical software computes: ${n+3\choose 3}-6{n+4\choose 4}+6{n+5\choose 5}=\sum_{k=0}^n{k+2\choose 2}^2$. The right-hand side is the value computed on the aforementioned paper.
\end{example}

\begin{example}

Calculate the Shuffle series of the poset associated to the
tree in Figure~\ref{fig:rt} assuming on the corresponding poset that $a$ is the minimum. The factorization of the poset is in Figure~\ref{Fig:Poset Tree}.
%-----------------------%
\begin{figure}
    \centering
    \includegraphics{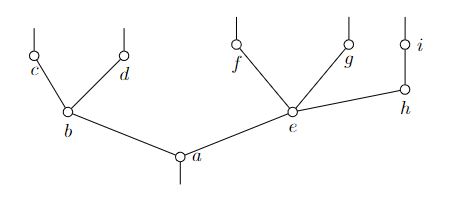}
    \caption{Reduced tree.}
    \label{fig:rt}
\end{figure}

\begin{figure}[h]\centering
  {\begin{tikzpicture}[xscale=1.3,yscale=1.3]
 \draw[step=1.0,black,thin] (-7,-3);% grid (3,3);
\draw[blue,fill] (-6.5,2.5) circle [radius=1pt]
node[left] {$c$};
\draw[color=green] (-6.35,2.4)  to (-5.55,1.6);

\draw[blue,fill] (-5.0,2.5) circle [radius=1pt]
node[right] {$d$};
\draw[color=green] (-5.0,2.4)  to (-5.4,1.6);

\draw[blue,fill] (-5.5,1.5) circle [radius=1pt]
node[left] {${\{c,d\}}$};
\draw[color=green] (-5.4,1.4)  to (-4.6,0.1);

\draw[blue,fill] (-4.5,0.0) circle [radius=1pt] node[left] {${\{d>c\}}$} ; 

\draw[blue,fill] (-3.5,1.5) circle [radius=1pt] node[left] {$b$} ; 
\draw[color=green] (-3.55,1.4)  to  (-4.4,0.1);

\draw[blue,fill] (-2.5,-2.1) circle [radius=1pt] node[left] {${\{c,d\}}$} ; 
\draw (-2.6,-2.0)  to  (-4.4,-0.1); 

\draw[color=magenta] (-2.4,-2.1)  to  (-0.1,-1.1); 

\draw[blue,fill] (0,-3.0) circle [radius=1pt] node[below] {${\{d>c\}}$} ;
\draw (-0.1,-2.95)  to  (-2.4,-2.2); 

\draw[blue,fill] (2.5,-2.0) circle [radius=1pt] node[right] {$a$} ; \draw (0.1,-3.0)  to  (2.35,-2.1); 

\draw[blue,fill] (-.5,2.5) circle [radius=1pt] node[right] {$g$};
\draw[blue,fill] (-1.5,1.5) circle [radius=1pt]
node[left] {${\{c,d\}}$};
\node[left] at (-0.5,0.5) {};

\draw[blue,fill] (2.5,2.5) circle [radius=1pt]
node[right] {$h$};

\draw[color=magenta] (1.5,1.5)  to (2.4,2.4);

\draw[blue,fill] (0,0) circle [radius=1pt]
node[right] {${\{c,d\}}$};
\draw[color=magenta] (0.1,0.0)  to  (1.3,1.35);
\draw[blue,fill] (1.39,1.4) circle [radius=1pt] node[left]{${\{d>c\}}$};

\draw[blue,fill] (3.0,0.65) circle [radius=1pt] node[above]{$e$};
\draw[color=magenta] (2.85,0.6)  to  (0.15,-1);

\draw[blue,fill] (0,-1) circle [radius=1pt] node[left]{${\{d>c\}}$};
\node[left] at (0.1,-0.5) {};
\draw[blue,fill] (1.39,1.4) circle [radius=1pt];
\node[left] at (2.5,-0.5) {};
\draw[blue,fill] (0.2,2.5) circle [radius=1pt]node[left]{$i$};
\draw[color=magenta] (0.3,2.4)  to  (1.35,1.5);

\draw[color=magenta] (-1.4,1.4)  to  (-.1,0.1); 
\draw[color=magenta] (0,-.1)  to  (0,-.9);
\draw[blue,fill] (-2.5,2.5) circle [radius=1pt] node[left] {$f$};

\draw[color=magenta] (-1.4,1.6)  to  (-.6,2.4);
\draw[color=magenta] (-1.6,1.6) to  (-2.4,2.4);

\end{tikzpicture}}
\caption{Factorization of the poset of edges of the tree in Figure~\ref{fig:rt}.\label{Fig:Poset Tree}
}
\end{figure}

 Since the action of $\widetilde{\{c<d\}}$ is commutative, for convenience we will write $\widetilde{\{d>c\}}$ in this Example.
The corresponding generating series is
\begin{align}
    \widetilde{\{d>c\}} \bigg[
                \widetilde{\{c,d\}}& \bigg\{ { 
                             \widetilde{\{d>c\}} \bigg(
                                         \widetilde{\{c,d\}}\big(
                                                       \mathcal{SH}(c), \mathcal{SH}(d)
                                                \big), \mathcal{SH}(b)
                                    \bigg),}
                    \widetilde{\{d>c\}} \bigg(\
                                            \widetilde{\{c,d\}}
                                            \notag\\
                                                    \big(
                                                        \widetilde{\{c,d\}}\left(
                                                                    \mathcal{SH}(f),\mathcal{SH}(g)
                                                               \right),
                                                    \big.
                                        \bigg.
                          \bigg.  
            \bigg.                                   
    & \bigg.
         \bigg.
             \bigg.{\widetilde{\{d>c\}}\left(
                               \mathcal{SH}(i), \mathcal{SH}(h)
                          \right),\mathcal{SH}(e)
            \bigg)}
        \bigg\},\mathcal{SH}(a)
    \bigg], \notag
\end{align}
 Here $a,b,c,d,e,f,g,h,i$ are copies of the poset $\chain[1]$. Using Equation~\ref{eqn:U} and Lemma~\ref{lemma:ch}, the evaluation along the green branch in Figure~\ref{Fig:Poset Tree} returns
\begin{align}
    {\widetilde{\{d>c\}}
        \big(\widetilde{\{c,d\}}
        \big.
    }&{
        \big.
            \left(\mathcal{SH}(1), \mathcal{SH}(1)
            \right)
        \big),
    \mathcal{SH}(1))}=  \widetilde{\{d>c\}}
        \bigg(\sum_{r=0}^{1}(-1)^{1+r}{1+r \choose 1}{1 \choose r} \mathcal{SH}(1+r),\mathcal{SH}(1)
        \bigg)
    \notag \\ 
     =& \widetilde{\{d>c\}} 
        \bigg(2 \mathcal{SH}(2)-\mathcal{SH}(1), \mathcal{SH}(1)
        \bigg) \notag \\
     =& 2 \mathcal{SH}(3)-\mathcal{SH}(2) \label{sh:11}
\end{align}
Now, the evaluation along the magenta branch in Figure~\ref{Fig:Poset Tree} returns:
\begin{equation}
{ \widetilde{\{d>c\}}\bigg\{\widetilde{\{c,d\}}
               \bigg(
                    \widetilde{\{c,d\}}\left(
                           \mathcal{SH}(1),\mathcal{SH}(1)
                           \right),  \widetilde{\{d>c\}}
                           \left(\mathcal{SH}(1), \mathcal{SH}(1)
                           \right)
                \bigg),\mathcal{SH}(1)
        \bigg\}} \label{sh:1}
\end{equation}

Using $
\widetilde{\{c,d\}} \left(\mathcal{SH}(1), \mathcal{SH}(1)\right) 
= 2 \mathcal{SH}(2)-\mathcal{SH}(1),$ and 
$\widetilde{\{d>c\}}\left(\mathcal{SH}(1), \mathcal{SH}(1)\right)= \mathcal{SH}(2) \label{sh:3}$  in Equation~\eqref{sh:1}, we have 
\begin{align}
    \widetilde{\{d>c\}}
    \bigg\{
        \widetilde{\{c,d\}}&
        \bigg(
            \widetilde{\{c,d\}}
            \left(
                \mathcal{SH}(1),\mathcal{SH}(1)
            \right), 
                \widetilde{\{d>c\}}
                \left(
                    \mathcal{SH}(1), \mathcal{SH}(1)
                \right)
        \bigg),\mathcal{SH}(1)
    \bigg\}
    \notag \\ 
    =& \widetilde{\{d>c\}}\bigg\{\widetilde{\{c,d\}}\bigg(\left(2\mathcal{SH}(2)-\mathcal{SH}(1), \mathcal{SH}(2)\right)\bigg), \mathcal{SH}(1)\bigg\} \notag  \\
    =& \widetilde{\{d>c\}}\bigg\{\bigg(\widetilde{\{c,d\}}\left(2\mathcal{SH}(2),\mathcal{SH}(2)\right)-\widetilde{\{c,d\}} \left( \mathcal{SH}(2), \mathcal{SH}(1)\right)\bigg), \mathcal{SH}(1) \bigg\} \notag \\
   =& \widetilde{\{d>c\}} \bigg(\left(12\mathcal{SH}(4)-12\mathcal{SH}(3)+2 \mathcal{SH}(2)-3 \mathcal{SH}(3)+2 \mathcal{SH}(2)\right), \mathcal{SH}(1)\bigg) \notag \\
  =& \widetilde{\{d>c\}} \bigg(\left(12 \mathcal{SH}(4)-15 \mathcal{SH}(3)+4 \mathcal{SH}(2)\right),  \mathcal{SH}(1)\bigg)\notag \\
  = & \widetilde{\{d>c\}} \left(12\mathcal{SH}(4), \mathcal{SH}(1)\right)-\widetilde{\{d>c\}}\left(15\mathcal{SH}(3), \mathcal{SH}(1)\right)+\widetilde{\{d>c\}}\left(4 \mathcal{SH}(2), \mathcal{SH}(1)\right)\notag \\
=& 12 \mathcal{SH}(5)-15 \mathcal{SH}(4)+4 \mathcal{SH}(3) . \label{sh:4}
\end{align}

Now, apply $\widetilde{\{c,d\}}$ on \eqref{sh:4} and \eqref{sh:11}
\begin{align}
\widetilde{\{c,d\}}
\bigg\{
    { \widetilde{\{d>c\}}} &{ 
        \bigg(
            \widetilde{\{c,d\}}
            \left(
                \mathcal{SH}(c), \mathcal{SH}(d)
            \right), \mathcal{SH}(b)
        \bigg)
            }
    ,  
    {
            \widetilde{\{d>c\}}
        \bigg(
            \widetilde{\{c,d\}}
            \big(
                \widetilde{\{c,d\}}
                \left(
                    \mathcal{SH}(f), \mathcal{SH}(g)
                \right)
    ,} \notag
    \\   
&     
{
    {
    \widetilde{\{d>c\}}
    \left(
        \mathcal{SH}(i), \mathcal{SH}(h)
    \right)
    \big), \mathcal{SH}(e)
    }
}
\bigg\}\notag \\
 =&\widetilde{\{c,d\}} 
    \bigg(
        {{2\mathcal{SH}(3)-\mathcal{SH}(2)}
        }, 
        {{12\mathcal{SH}(5)-15\mathcal{SH}(4)+4\mathcal{SH}(3)}
        }
    \bigg) \notag \\
=& 24 
    \widetilde{\{c,d\}}
    (\mathcal{SH}(5), \mathcal{SH}(3)
    )-30 
    \widetilde{\{c,d\}}
    (\mathcal{SH}(4),\mathcal{SH}(3)
    )+8
    \widetilde{\{c,d\}}
        (\mathcal{SH}(3),\mathcal{SH}(3)
        )
        -\notag \\
& 12\widetilde{\{c,d\}}
    (\mathcal{SH}(5), \mathcal{SH}(2)
    )+15
    \widetilde{\{c,d\}}
    (\mathcal{SH}(4),\mathcal{SH}(2)
    )-4
    \widetilde{\{c,d\}}
    (\mathcal{SH}(3), \mathcal{SH}(2)
    ) \notag \\
=& (1344\mathcal{SH}(8)-2625\mathcal{SH}(7)+1440\mathcal{SH}(6)-240\mathcal{SH}(5)
   )-
   (1050\mathcal{SH}(7)-1800\notag \\ 
   & \mathcal{SH}(6)+900\mathcal{SH}(5)-120\mathcal{SH}(4)
   )+
   (160\mathcal{SH}(6)-240\mathcal{SH}(5)+96\mathcal{SH}(4)-\notag \\ 
   & 8\mathcal{SH}(3)
   )-
   (252\mathcal{SH}(7)-360\mathcal{SH}(6)+120\mathcal{SH}(5)
   )+
   (225\mathcal{SH}(6)-300\mathcal{SH}(5)+\notag \\
   & 90\mathcal{SH}(4)
   )-(40\mathcal{SH}(5)-32\mathcal{SH}(4)+12\mathcal{SH}(3))\notag \\
= &  1344\mathcal{SH}(8)-3927\mathcal{SH}(7)+3985\mathcal{SH}(6)-1840\mathcal{SH}(5)+338\mathcal{SH}(4)-20\mathcal{SH}(3). \label{sh:5} 
\end{align}

Finally, we replace \eqref{sh:5} in the following equation
\begin{align}
    \widetilde{\{d>c\}} \bigg[
                \widetilde{\{c,d\}}& \bigg\{ { 
                             \widetilde{\{d>c\}} \bigg(
                                         \widetilde{\{c,d\}}\big(
                                                       \mathcal{SH}(c), \mathcal{SH}(d)
                                                \big), \mathcal{SH}(b)
                                    \bigg),}
                    {\widetilde{\{d>c\}} \bigg(\
                                            \widetilde{\{c,d\}}\big(
                                                        \widetilde{\{c,d\}}\left(
                                                                    \mathcal{SH}(f),\mathcal{SH}(g)
                                                               \right),
                                                    \big.
                                        \bigg.}
                          \bigg.  
            \bigg.                                    \notag\\
    & \bigg.
         \bigg.
             \bigg.{\widetilde{\{d>c\}}\left(
                               \mathcal{SH}(i), \mathcal{SH}(h)
                          \right),\mathcal{SH}(e)
            \bigg)}
        \bigg\},\mathcal{SH}(a)
    \bigg] \notag 
\end{align}
\begin{align}
\widetilde{\{d>c\}}&\bigg[
             1344\mathcal{SH}(8)-3927\mathcal{SH}(7)+3985\mathcal{SH}(6)-1840\mathcal{SH}(5)+338\mathcal{SH}(4)-20\mathcal{SH}(3),\mathcal{SH}(1)
        \bigg] \notag \\
=& 1344\mathcal{SH}(9)-3927\mathcal{SH}(8)+3985\mathcal{SH}(7)-1840\mathcal{SH}(6)+338\mathcal{SH}(5)-20\mathcal{SH}(4). \label{sh:6}
\end{align}
\end{example}

\section{The three algebras associated to an operation of posets}
\label{Sec:Second}

In algebraic topology, we replace a topological space $P$ with a ring $H^\ast(P)$ in order to use algebra and compute properties of the original space. Similarly, we start with a poset $P$ which we replace by the series  $\sum_{n=0}^\infty \#\hom_C(P,\chain[n])x^n$ where $C$ is some binary function we aim to study. We will show that if $C$ satisfies conditions $C_1$ and $C_{{2}}$, then to every poset we can associate three series. We think of these three series as shadows of posets, in the sense that the assignment poset to power series loses information, for example,{ in general, it is not injective and some noncommutative endomorphisms become commutative}. We also show that the operad of finite posets, which acts on posets, still acts on these three shadows of posets.

\subsection{The triple of order series}

\label{sec:new} A strict order-preserving map $f:P\rightarrow Q$ is a function that sends pairs $x,y\in P$ with $x<_P y$ to pairs satisfying $f(x)<_Q f(y)$. Stanley introduced the strict order polynomial $\somega[n]{P}$ (see Section~\ref{sec:exp}), it counts strict order-preserving maps from $P$ to $\chain[n]$ and we consider it as a polynomial on the variable $n$. Strict order series  $\zetaf[P]=\sum_{n=1}^\infty \somega[n]{P}x^n$ were introduced in Definition~\ref{series:strict}.
%Add definition

Weak order-preserving maps are poset maps with the requirement $x\leq y$ is sent to $f(x)\leq f(y)$.
 Stanley also introduced the weak polynomial $\nsomega[n]{P}$ which counts weak order-preserving maps from $P$ to $\chain[n]$. 

On strict order series, we work with the operad of finite posets $\mathcal{FP}$, elements in $\mathcal{FP}(n)$ are trees with $n$ leaves whose vertices with $k$ children are labeled by an indecomposable poset with $k$ points. 
%Correct this
Pick an element $\widetilde{Q}$ of the operad $ \mathcal{FP}(n)$. We pick a planar structure on the factorization of $\widetilde{Q}$, an ordering of the leaves, in order to evaluate $\widetilde{Q}$ on an $\mathcal{FP}$-algebra $A$.

Strict order series admit the structure of algebra over the operad $\mathcal{FP}$ (see Theorem~2.4 of \cite{APV}). In contrast, while we can define an action of the operad $\mathcal{FP}$ on shuffles series (as in Section~\ref{sec:exp}), we can only compute the action of the operad of $\mathcal{SP}$ on shuffle series. The issue is that we do not have an explicit formula for $\widetilde{\{x<y>z<w\}}$ and other operations defined by non series-parallel posets.  See Section~\ref{Sec:Third} for more details.

The papers~\cite{neg, LoebBin} by Loeb introduce sets with negative elements, see also~\cite{ newshuffle, SA, hybrid, gaussian}, \cite{NT} and \cite{SB} for further development and applications of the theory. The binomial symbol ${n\choose k}$ has six meanings based on the region in which the coefficients $n,k$ are located. In the {T}able~\ref{tab:my_label} we show the three regions in which the binomial returns an integer. 
On each region, the binomial counts the number of subsets (perhaps with negative elements) out of a given set (perhaps with negative elements).

\begin{table}[htb] \caption{Three interpretations of the binomial coefficient}\label{tab:my_label}% 
\begin{tabular}{@{}cc@{}}
%\toprule
    Value& Range of parameters\\ %\midrule
        ${n\choose k}$ & $ n\geq k\geq 0$ \\
        $(-1)^k\multiset{-n}{k}$ & $ k\geq 0 > n$ \\
        $(-1)^{n+k}\multiset{-n}{-(k-n)}$ & $ 0> n\geq k$ \\
%\botrule
\end{tabular}\end{table}

Stanley associated to a poset $P$ a polynomial with integer values $\somega[x]{P}=\sum_{i=1}^{|P|} d_{P,i} {x\choose i}$ were ${x\choose i}=\frac{x(x-1)\cdots(x-(i-1))}{i!}, d_{P,i}\in\mathbb{Z}$. From Loeb's result it follows that the 
polynomials:
$$\somega[x]{P}=\sum_{i=1}^{|P|} d_{P,i} {x\choose i},\sum_{i=1}^{|P|} (-1)^{|P|-i}d_{P,i} \multiset{x}{i},$$ 

have all rational coefficients and integer values. 
Note that the second polynomial is $(-1)^{|P|}\somega[-x]{P}$ due to  $(-1)^b{-a\choose b}=\multiset{a}{b}$.

\begin{theorem*}[\cite{ crtbook,beginning,crt}]The polynomial $(-1)^{|P|}\somega[-x]{P}$ coincides with $\nsomega[x]{P}$, in other words, it counts weak order-preserving maps from $P$ to $\chain[n]$. 
\end{theorem*}
This result is known as Stanley combinatorial reciprocity theorem.

 On the region $0>n\geq k$, the third interpretation of the binomial coefficient $\multiset{-n}{-(k-n)}={-k-1\choose -n-1}$ is not a polynomial on the variable $n$. Let $m=-n, i=-k$, since $\sum_{m=1}^{i}(-1)^{m-i}\multiset{m}{i-m}x^m=\sum_{m=1}^{i}(-1)^{m-i}{i-1\choose m-1} x^m=x(x-1)^{i-1}$, we can proceed formally at the level of order series. 

%Add definition
One computes 
\begin{equation}
    \zetaf[P]=\sum_{n=0}^\infty \somega[n]{P}x^n=\sum_{i=1}^{|P|} d_{P,i} \frac{x^i}{(1-x)^{i+1}}.\label{eqn:strict}
\end{equation}
\begin{definition}
    Given a poset $P$, we define the  weak and surjective weak order series:

\begin{equation}
 \zetafp[P]:=\sum_{n=0}^\infty \nsomega[n]{P}x^n=\sum_{i=1}^{|P|} (-1)^{|P|-i}d_{P,i} \frac{x}{(1-x)^{i+1}}, \label{eqn:weak}
 \end{equation}
 \begin{equation}\zetafs[P]:=\sum_{i=1}^{|P|}(-1)^{|P|-i} d_{P,i} x(1-x)^{i-1}.\label{Eqn:3series}  
\end{equation}
\end{definition}

Here we used the identities $\sum_{n=k}^\infty {n\choose k}x^n=\frac{x^k}{(1-x)^{k+1}}$ and 
$\sum_{n=1}^\infty \multiset{n}{k}x^n=\frac{x}{(1-x)^{k+1}}$ from~{\cite[Equation~(1.5.5)]{gf}, \cite[Equation~(1.3)]{lcomb} or \cite[Equation~(1.3)]{eulerianB}}. A combinatorial interpretation of the series $\zetafs[P]$ will be provided in Corollary~\ref{Cor:basis}.

\subsection{Ehrhart theory}
\label{Ehrhart:theory}
Let $P$ be a finite poset. The order polytope of $P$  $Poly(P)$ is defined as $$\{f:P\rightarrow [0,1]| f(x)\leq f(y)\hbox{ if }x\leq_P y\}.$$ For example if  $P={\{c<d\}}$ then $Poly(P)=\{(x,y)|0\leq x\leq y\leq 1\}$, it follows that $Poly(\chain[n])=\Delta_n$.

From now on we will use Ehrhart's theory point of view, see \cite{computingd, crtbook}. For any $k$ we consider the $k$-simplex $\Delta_k=\{(x_1,\cdots,x_k)|0\leq  x_1\leq\cdots\leq x_k\leq 1\}$. Denote by $\Delta_k^\circ$ the interior of the simplex. Given $n\in\mathbb{N}$ let $n\Delta_k$ be the expansion of the simplex by $n$ with vertices $\{(0,\cdots,0),(0,\cdots,0,n),\cdots, (n,\cdots,n)\}$. We denote by
$\mathcal{L}_k$ the lattice of  coordinates with natural indices in $\mathbb{R}^k$.

How many elements are in $\mathcal{L}_k\cap n\Delta_k$? Ehrhart proved that for any poset $P$, with order polytope $Poly(P)$ and number of points $|P|=k$, there is a polynomial $E(P)$ on the variable $n$ that counts lattice points on the $n$ expansion of $Poly(P)$, that is $E(P)(n)=\#\mathcal{L}_k\cap nPoly(P)$. Counting lattice points inside of an order polytope is related to the enumeration of maps of posets. To clarify this relation, note that the geometric object $Poly(P)$ is assumed to be inside $\mathbb{R}^k$, which has one coordinate per point in $P$. We then label the coordinates with the points of $P=(\{p_1,\cdots,p_k\},\leq_P)$ and the object $Poly(P)$ is the subset of the unit hyper-cube in $\prod_{i=1}^k\mathbb{R}<p_i>$ where the coordinates satisfy the inequalities of the poset $P$. Consider the set of integer coordinates inside $nPoly(P)$. Every such coordinate $(x_{p_1},\cdots,x_{p_n})$ defines a function with integer values $f:P\rightarrow \{0,1,\cdots, n\}$ given by $f(p_i)=x_{p_i}$. We then think of $f$ as a function $P\rightarrow \chain[n+1]$, and by the definition of $Poly(P)$ the function preserves the order. Then $\#\mathcal{L}_k\cap nPoly(P)\leq\nsomega[n+1]{P}$. On the other hand, every $f\in \nsomega[n+1]{P}$ defines a coordinate $(f(p_1)-1,\cdots,f(p_k)-1)$ and this assignment is injective as we have as many coordinates as points in the poset, then $\#\mathcal{L}_k\cap nPoly(P)=\nsomega[n+1]{P}$. In this case, we use the interpretation of the symbol  ${n\choose k}=(-1)^k\multiset{-n}{k}$ as counting subsets out of a negative subset, which allows for repetitions of elements.

Given a poset $P$, and a subset $I\subset P$, we say that $I$ is a lower set if $y\in I$ and $x\leq y$ implies $x\in I$. The poset $J(P)$ is the poset of lower sets ordered by inclusion.  
\begin{definition}\label{Def:J}
If $K=\{I_1\subset\cdots\subset I_k\}$ is a chain of lower sets with strict inclusions,  define 
\[F_K=\Bigg\{f:P\rightarrow \mathbb{R}:\begin{cases}
    {a)\, } f \hbox{ is constant on the subsets } I_1, I_2\setminus I_1,\cdots,I_k\setminus I_{k-1}, P\setminus I_k,\\
    {b)\, }0=f(I_1)\leq f(I_2)\leq \cdots\leq f(I_k)\leq f(P\setminus I_k)=1. \end{cases}\Bigg\}\]
    
Then, the canonical triangulation~~\cite[Section 5]{two} of the order polytope of $P$ is given by the simplicial complex $\{F_K|K\in J(P)\}$.\end{definition}

In order theory we aim to understand how an object is enriched by assigning an order to the object. In the case of order polytopes, they are polytopes equipped with a canonical triangulation. 

The following Lemma is proved in Lemma 2.3 of \cite{APV}.

\begin{lemma}[Inclusion Exclusion]\label{IEP} Let $P$ be a finite poset, $n=|P|$. If the canonical triangulation of the order polytope of $P$ is given by
 \[Poly(P)=\sqcup_{d_{P,|P|}} \Delta_{|P|}\setminus( \sqcup_{d_{P,|P|-1}} \Delta_{|P|-1})\sqcup ( \sqcup_{d_{P,|P|-2}} \Delta_{|P|-2})\setminus \cdots (\sqcup_{d_{P, 1}} \Delta_{1}),\]
then,  $$\zetaf[P]=\sum_{i=1}^n d_{P,i}\frac{x^i}{(1-x)^{i+1}},$$ 
\end{lemma}
%Add definition
\begin{definition}
We call vectors the sequence $(d_{P,1},\cdots,d_{P,|P|})$ defined in the Inclusion-Exclusion Lemma.
\end{definition}

The lemma follows from geometric arguments, we count points using the inclusion-exclusion property, that is, we are counting elements in subsets. 
%Add definition
\begin{definition}    
Given a poset $P$ and $n\in\mathbb{N},$ define $\hom_{Poset^s}(P,\chain[n])$ as the set of surjective weak poset maps.
\end{definition}

For example $\hom_{Poset^s}(\chain[3],\chain[2])=\{(1,2,3)\mapsto(1,1,2), (1,2,3)\mapsto(1,2,2)\}$. To define weakly surjective maps $\chain[k]\rightarrow\chain[n]$ we know the image must contain $1,\cdots, n$. Then $\multiset{n}{k-n}$ tells us in how many ways we can assign the remaining $k-n$ values. 
We compute $\sum_{n=1}^\infty(-1)^{n+k}\#\hom_{Poset^s}(\chain[k],\chain[n])x^n=\zetafs[k]$, where $\zetafs[k]$ is defined in \eqref{Eqn:3series}. 

Denote the interior of $nPoly(P)$ by $nPoly(P)^\circ$. It is shown in~\cite{MPoly} that the integer coordinates of $\mathcal{L}_k\cap nPoly(P)^\circ$, where $|P|=k$, correspond to strict order preserving maps, leading to $\#\mathcal{L}_k\cap nPoly(P)^\circ=\somega[n+1]{P}$. This case corresponds to the interpretation of the binomial coefficient as counting subsets out of a set.

In the same spirit, define $\mathcal{L}_{k,{ n}}$ as the set of coordinates in $\mathbb{Z}^k$ that include all digits from $0$ to $n$. For example $(0,0,1),(0,1,1)\in\mathcal{L}_{3,{ 1}}$.
We find $$\#\mathcal{L}_{k,n}\cap nPoly(P)=\#\hom_{Posets^s}(P,\chain[n+1]).$$

The action of the operad of finite posets on weak surjective series is of the form $$\widetilde{P}(\zetafs[P_1],\cdots,\zetafs[P_n])=\zetafs[\widetilde{P}(P_1,\cdots,P_n)].$$

The isomorphism from Loeb's work is compatible with the operadic structure of the series, the case of weak order series and strict order series is proved in Lemma 2.8 of  “A poset version of Ramanujan results on Eulerian numbers and zeta
values”.

\begin{corollary}[Operadic change of basis]
The change of basis $\{\frac{x}{(1-x)^{k+1}}\}\rightarrow \{{x}{(1-x)}^{k-1}\}$ is an operadic isomorphism of algebras over the operad $\mathcal{FP}$. \label{Cor:basis}
\end{corollary}
\begin{proof}
It follows since the action of the operad has the same vectors in both bases.
\end{proof}
Loeb introduced a map at the level of polynomials that sends $k\rightarrow -k$ on the second index of the binomial coefficients. 
The change of basis is meant to extend {the map defined by Loeb} on polynomials to a map on generating series. 

\begin{remark}
If $\phi$ is the change of basis of Corollary~\ref{Cor:basis}, then
\begin{eqnarray}
  \widetilde{\{c,d\}}(\zetafs[P],\zetafs[Q])&=& \phi^{-1}\big(\widetilde{\{c,d\}}\left(\phi(\zetafs[P]),\phi(\zetafs[Q])\right)\big). 
\end{eqnarray} The action of $\widetilde{\{c,d\}}$ on the basis $\{\zetafs[i]\}_{i\in\mathbb{N}}$ is similar to the Hadamard product on $\{\zetafp[i]\}_{i\in\mathbb{N}}$.   
\end{remark}

Corollary~\ref{Cor:basis} describes the structure of a set of series, we now explain the combinatorial meaning of the third series. 
\begin{lemma}\label{Lemma:CM} Let $P$ be a finite poset, then $\zetafs[P]= \sum_{n=1}^\infty(-1)^{n+1}\#\hom_{Poset^s}(P,\chain[n])x^n$.
\end{lemma}
\begin{proof}

We count $\#\mathcal{L}_{k,n}\cap nPoly(P)$ by first splitting the polytope $Poly(P)$ into simplices using the canonical triangulation of $Poly(P)$.
$$\#\hom_{Poset^s}(P,\chain[n+1])=\sqcup_{d_{P,|P|}} \#\mathcal{L}_{k,n}\cap n\Delta_{|P|}\setminus( \sqcup_{d_{P,|P|-1}} \#\mathcal{L}_{k,n}\cap n\Delta_{|P|-1})$$
$$\sqcup ( \sqcup_{d_{P,|P|-2}} \#\mathcal{L}_{k,n}\cap n\Delta_{|P|-2})\setminus \cdots (\sqcup_{d_{P, 1}} \#\mathcal{L}_{k,n}\cap n\Delta_1).$$

On the $\Delta_j$ we count $\#\hom_{Posets^s}(\chain[j],\chain[n])=\multiset{n}{j-n}$. 
 At the level of generating functions we obtain 
 \begin{eqnarray*}
 {\sum_{n=1}^\infty\#\hom_{Poset^s}(P,\chain[n])x^n}&=&
 \sum_{n=1}^\infty\sum_{i=1}^{|P|}d_{P,i}(-1)^{|P|-i}\#\hom_{Poset^s}(\chain[i],\chain[n])x^n\\
&=&
\sum_{i=1}^{|P|}(-1)^{|P|} d_{P,i}\sum_{n=1}^\infty (-1)^i\multiset{n}{i-n}x^n\\
&=&
\sum_{i=1}^{|P|}(-1)^{|P|}  d_{P,i}(-x)(-x-1)^{i-1}\\
&=&
\sum_{i=1}^{|P|}(-1)^{|P|-i-1}  d_{P,i}(-x)(1-(-x))^{i-1}
 \end{eqnarray*}
 The inclusion-exclusion principle implies the sign $(-1)^{|P|-j}$ on the first line. 
The lemma follows by evaluating on $-x$ and multiplying both sides by $(-1)$.
\end{proof}

\begin{corollary}
Let $P$ be a finite poset. Then, the number of weak surjective order preserving maps from $P$ to $\chain[s]$ is:
$$\sum_{i=s}^{|P|}(-1)^{|P|-i}{i-1\choose s-1}d_{P,i}.$$\label{Cor:vec}
\end{corollary}
\begin{proof}
 We compute the $s$ coefficient of $(-1)^{s+1}\zetafs[P]$:
\begin{eqnarray*}
\sum_{i=1}^{|P|} (-1)^{|P|-s}d_{P,i} \left( \sum_{j=0}^{i-1} {i-1\choose j}(-1)^{i-1-j}x^{j+1}\right)&=&\sum_{j=0}^{|P|-1}x^{j+1}\sum_{i=j+1}^{|P|}(-1)^{|P|-1-j-i-s}{i-1\choose j}d_{P,i}\\
&=&\sum_{k=1}^{|P|}x^{k}\sum_{i=k}^{|P|}(-1)^{|P|-k-i-s}{i-1\choose k-1}d_{P,i}
\end{eqnarray*}
\end{proof}

The family of series $\{\zetafs[P]\}_{P\in\hbox{ Posets}}$ corresponds to the interpretation of the symbol ${n\choose k}=(-1)^{n+k}\multiset{n}{k-n}$ counting subsets with a negative number of elements out of a subset with a negative number of elements. The change of basis Theorem is independent of the generalizations of the Stanley Reciprocity Theorem described in~\cite{crt} since it doesn't apply a reciprocity morphism. On the other hand, the change of basis Theorem and Lemma~\ref{Lemma:CM} are of a similar nature as the Stanley Reciprocity Theorem in the sense that they provide a combinatorial interpretation of series formally related.

\begin{remark}

We denote $\zetaf[\varnothing]:=1+x+\cdots$ since there is a unique map from the empty poset to any chain. From the theory on~\cite{posets} and \cite{APV} we have maps of $\mathcal{FP}$ algebras from order polytopes to order series. On polytopes, we have the operation of adding the boundary to an open polytope. The corresponding operation on power series sends strict order series to weak order series  $\zetaf[P]\mapsto\iota(\zetaf[P])=(-1)^{|P|+1}\zetaf[P]\circ i=\zetafp[P]$, where $i(x)=\frac{1}{x}$. How to represent the expansion of the order series of the empty set around infinity $\iota(\zetaf[\varnothing])=(-1)\zetaf[P]\circ i$? What symbol represents the closure of the empty set? we decided to use the zero Maya $\zetafp[\mayadigit{0}]:=\iota(\zetaf[\varnothing])$. Then  $\zetafp[\mayadigit{0}]$ is the unit in the  $\mathcal{FP}$ algebra of weak order series. \label{Remark:maya}
\end{remark}

With regards to the Yoneda lemma,  we lost information when we replaced $hom_{Posets}(P, )$ by the series $\sum_n \#hom_{Posets}(P,\chain[n])x^n$. Firstly, we restrict the environment by only considering those maps from $P$ to \textit{chains} instead of maps from $P$ to every other object on the category of ${SP}$ posets, the second part where we lost information is that we only \textit{count} those maps $\#hom_{Posets} (P,\chain[n])$. We believe the action of $\widetilde{\{c<d\}}$ becomes commutative because of this loss of information.

We aim to work with shuffle series, for that reason we return to work with algebras over the operad of series-parallel posets $\mathcal{SP}$. More precisely, the inclusion of series-parallel posets into the set of finite posets induces a map of operads: $\mathcal{SP}\rightarrow \mathcal{FP}$, then for an $\mathcal{FP}$-algebra $X$ we have a composition map: $\mathcal{SP}\rightarrow \mathcal{FP}\rightarrow End_X$ restricting the $\mathcal{FP}$-algebra structure on $X$ to an $\mathcal{SP}$-algebra structure. 

On weak surjective order series the action of $\widetilde{\{c<d\}}$ is forced to be $\frac{1-x}{x}$, to see this follow the proof of Lemma~\ref{lemma:ch}, first verify it on chains. 
%add definition
\begin{definition}
The $\mathcal{SP}$ algebra of strict order series is described by the following information: a generator, the unit over the operadic action, the action of $\widetilde{\{c,d\}}$, and the action of $\widetilde{\{c<d\}}$: 

\[(\zetaf[1]=\frac{x}{(1-x)^2},\zetaf[\varnothing]=\frac{1}{(1-x)}, \widetilde{\{c,d\}}=\hbox{Hadamard product of series, }\]

\[\widetilde{\{c<d\}}=\hbox{ deformed product of series by }(1-x)).\]

Strict order series corresponds to the interpretation of the symbol ${m\choose n}$ as counting subsets of a set.
Similarly, the $\mathcal{SP}$ algebra of weak order series is 

\[(\zetafp[1]=\frac{x}{(1-x)^2},\zetafp[\mayadigit{0}]=\frac{x}{(1-x)}, \widetilde{\{c,d\}}=\hbox{Hadamard product of series,}\]

\[\widetilde{\{c<d\}}=\hbox{ deformed product of series by }\frac{(1-x)}{x}),\]

and it corresponds to the interpretation of the symbol ${m\choose n}=(-1)^{n}\multiset{m}{n}$.
Finally, the $\mathcal{SP}$ algebra of weak surjective order series is given by: 

\[(\zetafs[1]={x},\zetafs[\mayadigit{0}]=\frac{x}{(1-x)}, \{c,d\}=\hbox{a similar operation to the Hadamard product}\]

\[\hbox{ of series, }\{c<d\}=\hbox{ deformed product of series by }\frac{(1-x)}{x}),\] 

this algebra corresponds to the interpretation of the symbol ${m\choose n}=(-1)^{n+m}\multiset{m}{n-m}$. The explicit expression for the action of $\widetilde{\{c<d\}}$ on weak surjective series follows from Corollary~\ref{Cor:basis}.
\end{definition}

\begin{remark}
A triangulated space $\{\mathcal{X}(n)\}_{n\in\mathbb{N}}$ is the geometric realization of a simplicial set (a functor $\mathcal{X}:\Delta^{op}\rightarrow Sets$)  that satisfies the following condition~\cite[Proposition 7]{methods}:  for any non-degenerate simplex $x\in \mathcal{X}_n$ and for any increasing map $f:[m]\rightarrow [n]$ the simplex $\mathcal{X}(f)(x)\in \mathcal{X}_m$ is also nondegenerate. In the first chapter of the reference cited above we learn that for a triangulated space we only need to specify the strict order-preserving morphisms among simplices to characterize the space, while a simplicial set requires weak order-preserving morphism. 

If we see order polytopes as triangulated spaces, then by the three interpretations of the binomial coefficient, to count all the weak order-preserving morphisms of order polytopes, it is enough to count all the weak surjective maps, which are surjective degenerations. 
\end{remark}

It is expected that the complexity of shuffles varies with the topology of the input posets.
 To motivate the next theorem, note that there is an isomorphism between shuffles of $\chain[k]$ and $\chain[m]$, and strict posets maps from $\chain[k]$ to $\chain[k+m]$.

\begin{theorem}As algebras over the operad $\mathcal{SP}$, the algebra of shuffle series is isomorphic to the algebra of weak order series.\label{main:thrm}\label{Remark:basis}

For any $\mathcal{SP}$-poset $P$, if $\zetafp[P]=\sum_{i=0}^n (-1)^{n-i}d_{P,i}\zetafp[i]$, then, the vectors of the corresponding shuffle series are $\mathcal{SH}(P)=\sum_{i=0}^n (-1)^{n-i}d_{P,i}\mathcal{SH}(i)$.
\end{theorem}
\begin{proof} 
We send the basis element $\zetafp[i]$ to $\mathcal{SH}(i)$.

From~\cite[Proposition 2.13]{posets}, it follows that $\{c<d\}(\zetafp[k],\zetafp[n])=\zetafp[k+n]$, and by~\cite[Remark 2.17]{posets} the action of $\{c,d\}$ has the same coefficients as Equation~\eqref{eqn:U}, then the action of the operad $\mathcal{SP}$ described on the basis $\{\zetafp[i]\}_{i\in\mathbb{N}}$ is the same as the action of the operad on the basis $\{\mathcal{SP}(i)\}_{i\in\mathbb{N}}$.
\end{proof}

A pair of posets $P,Q$ are Doppelg\"angers if $\zetafp[P]=\zetafp[Q]$.

\begin{corollary}\label{Inj}
    The $\mathcal{SP}$-posets $P$ and $Q$ are Doppelgangers iff $\mathcal{SH}(P)=\mathcal{SH}(Q)$.
\end{corollary}

This explains the computational observation from~\cite[Example 3.6]{ShuffleTree}.

In \cite{APV} %~\cite{posetzeta}
the structure of algebras over the operad of posets on Hurwitz zeta values $\{\zeta(n,k)\}_{n\in\mathbb{N}}$ is studied. It is known that multizeta values have a shuffle product, but our results imply that zeta values are already linked to the concept of shuffles: shuffle series and Hurwitz zeta values are algebras over the operad $\mathcal{SP}$.

\subsection{The triple of shuffles}
We remind the reader that in this section we are working with algebras over the operad $\mathcal{SP}$ instead of the operad $\mathcal{FP}$. 
%Verify we defined both operads explicitly
\begin{lemma}\label{lemma:3sh}

From Loeb's three interpretations of the binomial, we define three sets of series related to shuffles. Those sets have the following operadic structure:

\begin{itemize}
    \item The colimit-indexing shuffles $\mathcal{SH}(P)=\sum_{i=1}^{|P|} (-1)^{n-i}d_{P,i}\frac{1}{(1-x)^{i+1}},$ with the operadic structure:
$$(\mathcal{SH}(1)=\frac{1}{(1-x)^2}, \mathcal{SH}(\varnothing)=\frac{1}{1-x}, \{c,d\}=\hbox{Hadamard product},$$
$$\{c<d\}=\hbox{product by }(1-x)).$$ 
This algebra corresponds to the interpretation ${m\choose n}=(-1)^n\multiset{m}{n}$.

\item The right deck-divider shuffle series $dd\mathcal{SH}(P)=\sum_{i=1}^{|P|} (-1)^{|P|-i}d_{P,i}(1-x)^{i-1},$ with the operadic structure:
$$(dd\mathcal{SH}(1)=1, dd\mathcal{SH}(\varnothing)=\frac{1}{1-x}, \{c,d\}=\hbox{a similar operation to Hadamard product},$$ $$\{c<d\}=\hbox{product by }(1-x)).$$
This algebra corresponds to the intepretation ${m\choose n}=(-1)^{n+m}\multiset{m}{n-m}$. 

\item The left deck-divider shuffle series $ldd\mathcal{SH}(P)=\sum_{i=1}^{|P|} d_{P,i}\frac{x^{i+1}}{(1-x)^{i+1}},$ with the operadic structure:
$$(ldd\mathcal{SH}(1)=\frac{x^2}{(1-x)^2}, ldd\mathcal{SH}(\mayadigit{0})=\frac{x}{1-x}, \{c,d\}=\hbox{Hadamard product},$$
$$\{c<d\}=\hbox{product by }\frac{(1-x)}{x}).$$
\end{itemize}
\end{lemma}
\begin{proof}
We consider the series $\mathcal{SH}(P)$ as initial data.
The pairing (strict order series, weak order series, weak surjective order series) with (left deck-divider shuffle series, colimit-indexing shuffle series, right deck-divider shuffle series) follows Loeb's three interpretations of ${n\choose i}$. 

As sets, the other two series are obtained with the maps: $ldd\mathcal{SH}(P)=\iota(\mathcal{SH}(P))$, where $\iota$ is the operadic reciprocity morphism~\cite[Theorem 2.14]{posets}. The series $dd\mathcal{SH}(P)$ is defined from $\mathcal{SH}(P)$ via the change of basis $\{\frac{1}{(1-x)^{k+1}}\}_{k\in\mathbb{N}}\mapsto\{{(1-x)^{k-1}}\}_{k\in\mathbb{N}}$, see Corollary~\ref{Cor:basis}.

The operadic structure is first computed on chains and extended to series-parallel posets. 

Once we know that $\{\mathcal{SH}(P)\}_{P \in \mathcal{SP}-\text{poset}}$ as an algebra over the operad $\mathcal{SP} $ is isomorphic to the algebra of order series $\{\zetaf[P]\}_{P \in \mathcal{SP}-\text{poset}}$, the operadic structure of $\{ldd\mathcal{SH}(P)\}_{P \in \mathcal{SP}-\text{poset}}$ and $\{dd\mathcal{SH}(P)\}_{P \in \mathcal{SP}-\text{poset}}$ must coincide with the operadic structure of weak order series and weak surjective order series respectively, since those structures follow from combinatorial identities of the binomial coefficient.
\end{proof}

The previous theorem is algebraic in nature, it describes the underlying structure of sets. From the combinatorial point of view, what are each of those series counting?

The following two conditions generalize conditions $C_1$ and $C_2$:

\begin{itemize}
\item[$C_1^\prime$:] 
Every poset determines a vector in the space $\mathbb{Z}[Series_C(1),Series_C(2),\cdots]$.

\item[$C_2^\prime$:] The action of $\widetilde{\{c,d\}}$ on series of posets is an endomorphism of $\mathbb{Z}[Series_C(1),Series_C(2),\cdots]$: $$\widetilde{\{c,d\}}(Series_C(j), Series_C(k))=\sum_{i\in I} a_i Series_C(i), a_i\in\mathbb{Z},$$ with $I$ a finite set of indices.
\end{itemize}

\begin{lemma}\label{Lemma:conditions}
If a rule to generate posets $C$ from two inputs satisfies:
\begin{itemize}
\item there are associative operations $\widetilde{\{c,d\}}, \widetilde{\{c<d\}}$ defined on the space $\mathbb{Z}[Series_C(1),Series_C(2),\cdots]$,  
\item conditions $C_1^\prime$ and $C_2^\prime$,
\item the operation $\widetilde{\{c,d\}}$ is commutative,
\item both operations enumerate structures generated by $C$, 
\end{itemize}
  we conclude that there is a well defined action of the operad $\mathcal{SP}$. The series $$\widetilde{P}(Series_C(1),\cdots,Series_C(1))$$ enumerates $f\in\hom_C(P,\chain[n])$ for every $SP$-poset $P$. 
\end{lemma}
\begin{proof}
It follows by using the same arguments as in the proof of Lemma~\ref{lemma:WD} and Theorem~\ref{Thm:equality}. 
\end{proof}
%Add definition
\begin{definition}
Let $P$ be a $SP$-poset and fix a chain $\chain[n]$. Consider the Hasse diagram of the poset $P$ and draw $n$ horizontal lines in such a way that before the first line and after the last line there is at least one point of the poset $P$, and the lines are not consecutive. Call each of this possible configurations a right deck-divider shuffle.  
\end{definition}
 In Figure~\ref{ddmathcal{SH}} we see two examples of deck-divider shuffles.

\begin{figure}[thb]
\centering
\includegraphics[width = .5\textwidth]{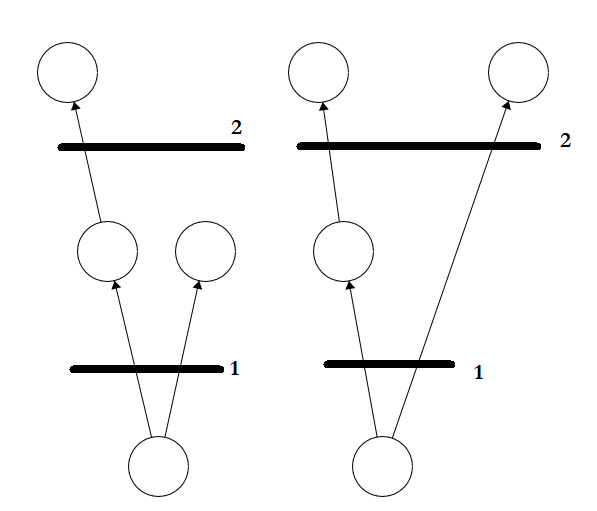}
\caption{Two examples of right deck-divider shuffles between a poset $\widetilde{\{c<d\}}(\chain[1] ,\widetilde{\{c,d\}}(\chain[2],\chain[1]))$ and $\chain[2]$}\label{ddmathcal{SH}}
\end{figure}

Given a poset $P$, a right deck divider shuffle of $P$ and $\chain[n]$ can be seen as colimit-indexing shuffles where there is a maximal chain in which the maximum point belong to the poset $P$, there is a maximal chain in which the minimum point belong to the poset $P$ for any two elements of $\chain[n]$ $i,j$ there is a maximal chain in which a point of $P$ separates them.
%poset by connecting a line with the points above the line, and connecting the line with the points below the line. Then we collapse the horizontal lines into points, and define the poset as the poset associated to the resultant Hasse diagram. 
\begin{theorem} Let $P$ be a $SP$-poset, then the $n$-coefficient of $(-1)^n dd\mathcal{SH}(P)$ counts right deck-divider shuffles between $P$ and the $\chain[n]$.
\end{theorem}
\begin{proof}

It is enough to prove that  $\widetilde{\{c,d\}}, \widetilde{\{c<d\}}$ satisfy Lemma~\ref{Lemma:conditions} for $C$ being right deck-divider shuffles. 

By construction condition $C_1^\prime$ is satisfied, and since the action of $\widetilde{\{c,d\}}$ on colimit-indexing shuffles satisfies $C_2$, then right deck-divider shuffles satisfy condition $C_2^\prime$. 

We find that the $n$-coefficient of  $dd\mathcal{SH}(i)=(1-x)^{i-1}$ count right deck-divider shuffles between the $i$ chain and the $n$ chain, see~\cite{Selcomb}. It follows that the $n$-coefficient of $(1-x)^{i+j-1}$ count right deck-divider shuffles of $\widetilde{\{c<d\}}(\chain[i],\chain[j])$ and $\chain[n]$.
 
The action $\widetilde{\{c,d\}}$ on two chains $(dd\mathcal{SH}(i),dd\mathcal{SH}(j)), i\geq j$ is defined by $$\sum_{k=0}^{j}(-1)^{j-k}{i+k\choose j}{j\choose k}(1-x)^{i+k-1}$$ (due to the operadic structure).

By inspection both actions are associative and $\widetilde{\{c,d\}}$ is commutative.
 
 Assuming $i\geq j$, there are ${i+j\choose j}{j\choose j}$ different linearizations of $\chain[i]$ and $\chain[j]$. For each one of those linearizations $A$, we have ${i+j-1\choose n}$ right deck-divider shuffles between $A$ and the $n$ chain. Assume that between the line $r$ and $r+1$ there is a point $a$ of $\chain[i]$ and a point $b$ of $\chain[j]$. In the linearization, either $a<b$ or $b<a$, but in the definition of the right deck-divider shuffle, we only care that both elements are between $r$ and $r+1$. That means that we are  double counting terms.

  To understand $\widetilde{\{c,d\}}(dd\mathcal{SH}(i),dd\mathcal{SH}(j))$, in each linearization we add equal signs between consecutive elements if one of them is labeled by an element in $\chain[i]$ and the other is labeled by an element in $\chain[j]$. In~\cite[Lemma 3.5]{posets} it is shown that we can add $j-h$ equal signs in ${i+h\choose j}{j\choose h}$ ways, and we can form  ${i+h-1\choose n}$ different right deck-divider shuffles. By inclusion-exclusion we obtain:
$\sum_{h=0}^i (-1)^{i-h}{i+h\choose j}{j\choose h}{i+h-1\choose n}$ different right deck-divider shuffles between $\widetilde{\{c,d\}}(\chain[i],\chain[j])$ and $\chain[n]$. In other words, $\widetilde{\{c,d\}}(dd\mathcal{SH}(i),dd\mathcal{SH}(j))$ preserves the combinatorial interpretation.

 Then, by Lemma~\ref{Lemma:conditions} the computation $\widetilde{P}(dd\mathcal{SH}(1),\cdots,dd\mathcal{SH}(1))$ preserves the combinatorial definition for every $SP$-poset $P$.  
\end{proof}

In Figure~\ref{ddmathcal{SH}} left we obtain the poset $\{x<1<y<2<z, 1<w<2\}$ and on the right we obtain the poset $\{x<1<y<2<z, 2<w\}$.

\begin{definition}    
For a series parallel poset $P$, define a left deck-divider shuffle between $P$ and $\chain[n]$ to be a colimit-indexing shuffle $A$ in which for every maximal chain $m\subset A$, the maximum and minimum points of $m$ are points from $\chain[n]$, and there are no two consecutive points of $P$ on $m$.\end{definition}

\begin{lemma}Let $P$ be a ${SP}$-poset, then the $n$-coefficient of $ldd\mathcal{SH}(P)$ counts left deck-divider shuffles.\label{Lemma:lsh}
\end{lemma}
\begin{proof}

To understand the action of an arbitrary $SP$ poset $P$ it is enough to study the action of $\widetilde{\{c,d\}}, \widetilde{\{c<d\}}$ on chains, because conditions $C_1$ and $C_2$ are satisfied.

In~\cite{Selcomb} it is shown that ${n+1\choose k}$ is the number of shuffles between $k$ and $n$ with no two elements of $k$ appearing consecutively. When we fix the first and last term to be elements of $n$ then ${n-1\choose k}$, the $n$ coefficient of $ldd\mathcal{SH}(k)$, counts the number of shuffles between $\chain[k]$ and $\chain[n]$ in which the first and last point are from $\chain[n]$ and there are no two consecutive points of $\chain[k]$. 

For chains we count ${n-1\choose i+j}$ left deck dividing shuffles between $\widetilde{\{c<d\}}(\chain[i],\chain[j])$ and $n$. Then $\widetilde{\{c<d\}}(ldd\mathcal{SH}(i),ldd\mathcal{SH}(j))$ preserves the combinatorial meaning.

Now the operation $\widetilde{\{c,d\}}(ldd\mathcal{SH}(P),ldd\mathcal{SH}(Q))$ considers pairs of shuffles that don't interact, and using \eqref{Eq:id} we conclude that the conditions of left deck-divider shuffles are preserved under both operations $\widetilde{\{c,d\}}, \widetilde{\{c<d\}}$. Then $\widetilde{P}(ldd\mathcal{SH}(1),\cdots,ldd\mathcal{SH}(1))$ counts left deck-divider shuffles of an $SP$-poset $P$ and chains.  
\end{proof}

In Figure~\ref{Fig:LsH} we display an example of a left deck-divider shuffle between a poset $\{a<c,a<d, b<c, b<d\}$ and $\chain[4]$. 

\begin{figure}
\begin{center}
\begin{tikzpicture}[scale=0.09]
\tikzstyle{every node}+=[inner sep=0pt]
\draw [black] (2.6,-44.8) circle (3);
\draw [black] (2.6,-53.7) circle (3);
\draw [black] (2.6,-53.7) circle (2.4);
\draw [black] (2.9,-4.9) circle (3);
\draw [black] (2.9,-4.9) circle (2.4);
\draw [black] (2.9,-14.8) circle (3);
\draw [black] (17.3,-35.8) circle (3);
\draw [black] (17.3,-14.8) circle (3);
\draw [black] (17.3,-53.7) circle (3);
\draw [black] (17.3,-53.7) circle (2.4);
\draw [black] (17.3,-44.8) circle (3);
\draw [black] (17.3,-44.8) circle (2.4);
\draw [black] (2.9,-25.3) circle (3);
\draw [black] (2.9,-25.3) circle (2.4);
\draw [black] (17.3,-24.4) circle (3);
\draw [black] (17.3,-24.4) circle (2.4);
\draw [black] (17.3,-4.9) circle (3);
\draw [black] (17.3,-4.9) circle (2.4);
\draw [black] (2.9,-35.8) circle (3);
\draw [black] (2.9,-35.8) circle (2.4);
\draw [black] (10.4,-35.8) circle (3);
\draw [black] (10.4,-35.8) circle (2.4);
\draw [black] (2.6,-50.7) -- (2.6,-47.8);
\fill [black] (2.6,-47.8) -- (2.1,-48.6) -- (3.1,-48.6);
\draw [black] (17.3,-50.7) -- (17.3,-47.8);
\fill [black] (17.3,-47.8) -- (16.8,-48.6) -- (17.8,-48.6);
\draw [black] (17.3,-41.8) -- (17.3,-38.8);
\fill [black] (17.3,-38.8) -- (16.8,-39.6) -- (17.8,-39.6);
\draw [black] (17.3,-32.8) -- (17.3,-27.4);
\fill [black] (17.3,-27.4) -- (16.8,-28.2) -- (17.8,-28.2);
\draw [black] (14.88,-34.03) -- (5.32,-27.07);
\fill [black] (5.32,-27.07) -- (5.68,-27.94) -- (6.27,-27.13);
\draw [black] (17.3,-21.4) -- (17.3,-17.8);
\fill [black] (17.3,-17.8) -- (16.8,-18.6) -- (17.8,-18.6);
\draw [black] (2.9,-22.3) -- (2.9,-17.8);
\fill [black] (2.9,-17.8) -- (2.4,-18.6) -- (3.4,-18.6);
\draw [black] (17.3,-11.8) -- (17.3,-7.9);
\fill [black] (17.3,-7.9) -- (16.8,-8.7) -- (17.8,-8.7);
\draw [black] (2.9,-11.8) -- (2.9,-7.9);
\fill [black] (2.9,-7.9) -- (2.4,-8.7) -- (3.4,-8.7);
\draw [black] (4.56,-42.53) -- (8.44,-38.07);
\fill [black] (8.44,-38.07) -- (7.53,-38.34) -- (8.29,-39);
\draw [black] (2.7,-41.8) -- (2.8,-38.8);
\fill [black] (2.8,-38.8) -- (2.27,-39.58) -- (3.27,-39.61);
\draw [black] (11.6,-32.8) -- (15.62,-26.88);
\fill [black] (15.62,-26.88) -- (14.75,-27.26) -- (15.58,-27.83);
\draw [black] (2.9,-32.8) -- (2.9,-28.3);
\fill [black] (2.9,-28.3) -- (2.4,-29.1) -- (3.4,-29.1);
\end{tikzpicture}
\end{center}
\begin{caption}
{A left deck-divider shuffle between $\{a<c,a<d,b<c,b<d\}$ and $\chain[4]$. The elements of the chain are shown with double circles.\label{Fig:LsH}}
\end{caption}
\end{figure}

For weak order series, one has the explicit construction of~\cite[Proposition 2.13]{posets} in which we take two weak order preserving maps $P\rightarrow \chain[n]$, $Q\rightarrow \chain[m]$, and identify the maximum of $\chain[n]$ with the minimum of $\chain[m]$ in order to define the map with inputs $f:P\rightarrow \chain[n], g:Q\rightarrow\chain[m]$ and output $f:\widetilde{\{c<d\}}(P,Q)\rightarrow \chain[n+m-1]$. For weak order series the action of $\widetilde{\{c,d\}}$ is multiplication by $\frac{x}{1-x}$. The same product occurs on left deck-divider shuffles. The operation $\widetilde{\{c<d\}}$ evaluated in left deck-divider shuffles is gluing shuffles: Let $r$ be a left deck-divider shuffle between $P$ and $\chain[n]$, and $t$ be a left deck-divider shuffle between $Q$ and $\chain[s]$. On each maximal chain $m_r$ of $r$ we glue a copy of  a maximal chain $m_t$ of $t$ by identifying the maximum of $m_r$ (which is the maximum of $\chain[n]$) with the minimum of $m_t$ (which is the minimum of $\chain[s]$). When all those glued chains can be glued to form a poset (a colimit condition), the poset is a left deck-divider shuffle in $\widetilde{\{c<d}\}(r,t)$.

Note that right deck-divider shuffles and their left version are not symmetric: the number of right deck-divider shuffles between $\chain[a]$ and $\chain[b]$ is different than the number of such shuffles between $\chain[b]$ and $\chain[a]$.

\subsection{Properties of the vectors}

\label{Sec:Third}

In this final section, we combine all the theory to describe properties of the vectors $ d_{P,i}$. We also link our work with the tensor product of operads of trees. 

The diagram in Figure~\ref{fig:my_tree} displays those algebras over the operad of finite posets that we have studied in~\cite{posets}, \cite{APV}  and the current paper. Every vertical arrow goes from top to bottom, and it is surjective. When we restrict all those algebras to the operad of series-parallel posets, then from Theorem~\ref{main:thrm} the algebra of weak order series is isomorphic to shuffle series. 

\begin{figure}[htb]
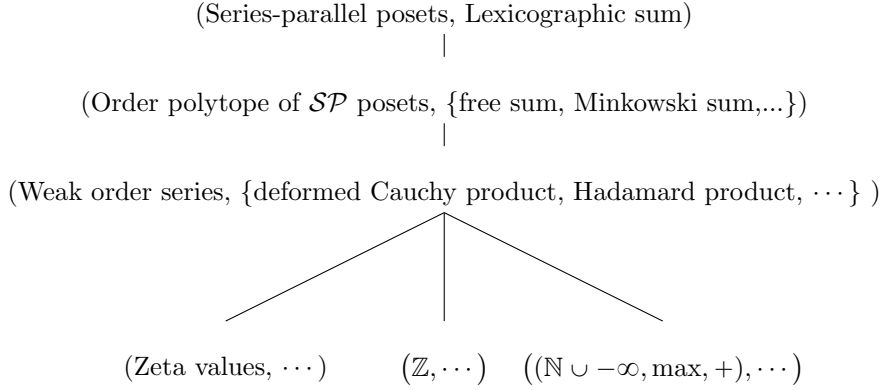
 
    \centering
    \Tree [.{(Series-parallel posets, Lexicographic sum)} 
        [.{(Order polytope of $\mathcal{SP}$ posets, \{free sum, Minkowski sum,...\})} 
             [.{(Weak order series, \{deformed Cauchy product, Hadamard product, $\cdots$\} )}
                 {(Zeta values, $\cdots$)}
            {$\big(\mathbb{Z},\cdots\big)$} 
            {$\big((\mathbb{N}\cup-\infty,\max,+),\cdots\big)$}
             ] 
         ]
      ]
    \vspace{1.5cm}
\caption{Relationship between algebras over the set operad of finite parallel posets, that are generated by one point.    \label{fig:my_tree}
}
\end{figure}
The first arrow associates to a poset its order polytope. The operation $\widetilde{\{c<d\}}$ acts on polytopes as the convex envelope of the inputs, and $\widetilde{\{c,d\}}$ is an orthogonal Minkovski sum, in which the inputs are assumed one in $R^m\subset R^{m+n}$ and the other in $R^n\subset R^{m+n}$.
The second arrow associates to a polytope its shifted Ehrhart series, equivalently, the order series of the underlying poset. 

The next arrow uses the $d_{P,i}$ coefficients of the series to define a set of numbers closely related to zeta values (see \cite{APV} 
for details). Given a vector $(d_{P,1},\cdots,d_{P,|P|})$ denote by $d_{P,i_0}$ the first non zero value. The last two arrows in the diagram map the vector $(d_{P,1},\cdots,d_{P,|P|})$ to $d_{P,|P|}$ and $d_{P,i_0}$ respectively. Here $(\mathbb{N}\cup-\infty,\max,+)$ are the integers inside the tropical numbers, see~\cite{WTP}, and Example 2.10 of \cite{APV}. 

Stanley explained that $d_{P,i}$ counts strict surjective maps from $P$ to $\chain[i]$, see~\cite{beginning}. Given a poset $P$, an order ideal is a subset $I$ so that $x\in I$ and $y\leq x$ then $y\in I$. Consider $J(P)$ the poset of order ideals of $P$/lower sets, ordered by inclusion.  Let $\mathcal{I}_P$ be the set of maximal chains in $J(P)$. If two elements $x, y$ belong to an antichain in $P$ then we can find two maximal chains, $m_{x<y}$ in which $succ(x)=y$ and $m_{y<x}$ in which $succ(y)=x$, and every other point coincide on the two chains. We declare the common boundary of those two maximal chains $m_{x<y}\prod m_{y<x}$ to be the poset with the intersection of those two linear orders, and define $|m_{x<y}\prod m_{y<x}|$ to be the height of the poset. On the diagram $\mathcal{I}_P$ we introduce the inclusion arrows: $$m_{x<y}\prod m_{y<x}\rightarrow m_{x<y},$$$$m_{x<y}\prod m_{y<x}\rightarrow m_{y<x}.$$

We iterate this process, if we constructed $m,m_2$ then they have a common boundary if there is a poset $m\prod m_2$ where there are exactly two points $x,y$ which are an antichain in $m\prod m_2$ but on $m$ we have $x<y$ and on $m_2$ we have $y<x$. The orders in $m, m_2, m\prod m_2$ coincide in any other couple of points. To $\mathcal{I}_P$ we add 
$m\prod m_2$ with the inclusion arrows:
$$m\prod m_2\rightarrow m,$$
$$m\prod m_2\rightarrow m_2.$$

If we also add identity maps then $\mathcal{I}_P$ becomes a category. We think of $\mathcal{I}_P$ as a diagram category on which we can apply functors, for example  $Poly(P)= colim_{I\in\mathcal{I}_P, |I|=n} |\Delta(n)|$ by construction.

The third level of Figure~\ref{fig:my_tree}, only remembers the statistics $d_{P,i}$ associated to a poset.

\begin{lemma} \label{Lemma:Colim}
Let $P$ be a finite poset. Then $d_{P,i}=\#\{I\in\mathcal{I}_P||I|=|P|-(i-1)\}$.

\end{lemma}\begin{proof} 

The simplices in the cannonical triangulation of $Poly(P)$ are indexed by maximal chains in $J(P)$, and their common boundaries~\cite{two}. The results follows by Lemma~\ref{IEP}. 
\end{proof}

At the combinatorial level of Figure~\ref{fig:my_tree} the vectors lost too much information as we can see by the existence of Doppelg\"angers. In contrast, at the second level we can extract the diagram $I_P$ that encodes how to glue the chains. 

\begin{lemma} For every finite poset $P$ we have $P= colim_{I\in\mathcal{I}_P, |I|=n} \chain[n]$.\end{lemma}
\begin{proof}
Let $Q$ be a finite poset and consider maps $f_I:I\rightarrow Q$ for all $I\in \mathcal{I}_P$.

For every $x\in P$ pick a maximal chain $I_x\subset J(P)$ that contains $x$, define $i(x):=f_{I_x}(x)$. If $x\in P$ and $y\in Succ(x)$; then, there is at least one maximal chain $I_{x<y}$ in which $x<y$. In fact, the inequality holds in every maximal chain. 
Since $f_{I_{x<y}}$ preserves the order, and since $f_{I_{x<y}}(x)=f_{I_{x}}(x), f_{I_{x<y}}(y)=f_{I_{y}}(x)$ then $i(x)\leq i(y)$. This means that $i:P\rightarrow Q$ preserves the order. By the universal property of the colimit,  $colim_{I\in\mathcal{I}_P, |I|=n} \chain[n]= P,$ where the maps from the colimit to the poset are inclusion.
\end{proof}

If follows that $dim(P)\leq d_{P,|P|}$, as the dimension of a poset is the minimun number of linear orders whose intersection define the poset. For $SP$-posets the dimension is 2, while $d_{P,|P|}$ is arbitrarily large.

\begin{example}
Consider the poset $P=\{x, q<r,q<s\}$. Since $\mathcal{SH}(P)=8\mathcal{SH}(4)-9\mathcal{SH}(3)+2\mathcal{SH}(2)$, then $J(P)$, the poset of ideals/lower sets of $P$,  has $8$ maximal ideals: 
\begin{eqnarray*}
x<q<r<s,& x<q<s<r,\\
q<x<r<s,& q<x<s<r,\\
q<r<x<s,& q<s<x<r,\\
q<r<s<x,& q<s<r<x.\\    
\end{eqnarray*}

Now, for a pair of elements $a,b$ in antichains, elements that are not related, both possibilities $succ(a)=b$ and $succ(b)=a$ occur in some linear order/shuffle. The intersection of two linear orders, that only differ on which of the two elements $a,b$ is successor of the remaining, is a poset with length 3:

\begin{eqnarray*}
\{x,q\}<r<s,& \{x, q\}<s<r,\\
q<\{x,r\}<s,& q<s<\{x,r\},\\
q<r<\{x,s\},& x<q<\{x,s\}<r,\\
x<q<\{r,s\},& q<x<\{s,r\},\\
q<\{s,r\}<x,&
\end{eqnarray*}

we already know there are 9 of them, as the coefficient of $\mathcal{SH}(3)$ counts those posets.

Finally, we have 2 posets obtained after intersecting three chains in $J(P)$:

\begin{eqnarray*}
\{x,q\}<\{r,s\},& q<\{x,r,s\}.
\end{eqnarray*}
\end{example}

From Ehrhart's theory point of view, the Ehrhart polynomial of $Poly(P)$ counts lattice points inside expansions of the polytope. The interpretation of the vectors $(d_{P,1},\cdots,d_{P,n})$ of the order series $\zetafp[P]$ is that they count simplices in the canonical triangulation of $Poly(P)$, see~Lemma~\ref{IEP}. For example, in the case of $P=\{x, q<r,q<s\}$ we have $8$ $4$-simplices and $9$ pairwise intersections of maximal simplices on $3$-simplices. There are $2$ triple intersections of maximal simplices on $2$-simplices. 

The Ehrhart polynomial of $Poly(P)$ determines the order series of $P$ and vice versa. Then, knowing the number of lattice points inside every expansion of a polytope is equivalent to knowing the statistics of the canonical triangulation of the polytope. To be precise, the statistics that are preserved are all the numbers $d_{P,i}=\#\{I\in\mathcal{I}_P||I|=i\}$.

\begin{lemma}\label{lem:11} Let $P$ be a finite poset. Then $d_{P,i}=\#\{I\in\mathcal{I}_P,|I|=i\}$  counts linear orders on the poset obtained after we identify $|P|-i$ pairs of variables of the poset that are not related under $\leq_P$.
\end{lemma}\begin{proof}
The coordinates of the simplices in the order polytope are described by linearizations of the poset.
The intersection of two maximal simplices is a simplex of lower dimension describing the space in which two coordinate values coincide, that is the coordinates satisfy $x<y$ on one simplex and $x>y$ on the other. Then those boundary links correspond to posets in which points are identified. The result follows from Lemma~\ref{IEP}.

\end{proof}

We now introduce notation from~\cite[Chapter 1.1, Chapter 1.6, Chapter 4.2]{SandDHT}, see the original references for details. A colored operad $P$ consist of a set $C$ of colours and for each $n\geq 0$, $c_1,\cdots,c_n,c$ colours, a set $P(c_1,\cdots,c_n;c)$ of operations, satisfying structural relations. We also assume actions of $S_n$ on the operations representing permutations of the inputs, those actions should satisfy some compatibility conditions.

When the set of colors is a singleton, we denote $P(c;c)$ by $P(1)$, $P(c,c;c)$ by $P(2)$, etc. The operad of finite posets is an example of an operad with one color.

 A tree $T$ defines an operad whose colors are the edges of $T$, given a sequence of colors $e_1, \cdots, e_n, e$ the set of operations $\Omega(e_1,\cdots, e_n;e)$ is empty or it is a point. The nontrivial case occurs when there is a subtree of $T$ with leaves $e_1,\cdots, e_n$, and root $e$. The category $\Omega$ has as objects trees, and operadic maps as morphism, in other words, a morphism sends edges to edges and vertices to subtrees. 

Let $T$ be a reduced tree, the result of pruning a tree and we put one leaf on each external vertex. Let $U(T)$ be the poset of vertices of $T$. What is the interpretation of $d_{T,i}$ in the study of $\Omega(T)$?

First, we take the point of view of linearizations of posets. To put a linear order on the vertices of $U(T)$ is equivalent as to choose a path along a maximal simplex in $Poly(U(T))$ covering all vertices and preserving their order. If one aims to implement any maximal composition of functions in a program, in such a way that one function is evaluated at a time, as in the Reverse Polish Notation, then we need to ensure we order the functions in a compatible way to the order of the vertices of $\Omega(T)$, for example, if function $e$ needs the output of function $f$ then function $f$ is evaluated first.
To an order preserving path along each vertex of a maximal simplex in the canonical triangulation of $U(T)$, we associate sequences of composition of operations of $\Omega(T)$. 

For example, consider the tree with poset of vertices $\{ q<r,q<s\}$. We think of ${q}$ as a binary operation and ${r,s}$ as unary operations.
The order polytope of the poset of vertices can be reparametrized to look like the cone of a square, a pyramid with square base $\{(r,s,q)| 0\leq q,r,s\leq1, q\leq \min\{r,s\}\}$. The path $(0,0,0)-(1,0,0)-(1,1,0)-(1,1,1)$ is then indicating the sequence of functions ${r}$, ${s}$, ${q}$. 
Assume you have two elements of an antichain in the poset of vertices of $T$, an $s$-ary operation and a $t$-ary operation. Then, they can be considered an $s+t$-ary operation. The number of linearizations where we allow $k$ such new operations is $d_{P,|P|-k}$.

Given two operads $\mathcal{P},\mathcal{Q}$, with sets of colors $C,D$ respectively, the tensor product of the operads $\mathcal{P}\otimes \mathcal{Q}$ is the operad with colors $c\otimes d$ for $c\in C, d\in D$. The operations are generated by the following rules:
\begin{itemize}
\item for each $p\in\mathcal{P}(c_1,\cdots,c_n;c)$ and colour $d\in D$ of $\mathcal{Q}$, there is an operation $$p\otimes d\in (c_1\otimes d,\cdots c_n\otimes d;c\otimes d),$$ 
\item for each colour $c\in C$ of $\mathcal{P}$
 and operation $q\in \mathcal{Q}(d_1,\cdots,d_m;d)$, there is an operation $$c\otimes q\in (c\otimes d_1,\cdots c\otimes d_m;c\otimes d).$$ 
\end{itemize}
We require that for any color $d\in D$ of $\mathcal{P}$ the map $-\otimes d:\mathcal{P}\rightarrow \mathcal{P}\otimes \mathcal{Q}$ is an operadic map,  for any color $c\in C$ of $\mathcal{Q}$ the map $c\otimes-:\mathcal{Q}\rightarrow \mathcal{P}\otimes \mathcal{Q}$ is an operadic map, and the Boardman-Vogt interchange relation: $$p\otimes d(c_1\otimes q,\cdots,c_n\otimes q)=\sigma_{n,m}^\ast\left((c\otimes q)(p\otimes d_1,\cdots,p\otimes d_m)\right),$$
here the permutation $\sigma_{n,m}$ relates the sequences $(c_1\otimes d_1,\cdots, c_1\otimes d_m,\cdots,c_n\otimes d_1,\cdots,c_n\otimes d_m)$ and $(c_1\otimes d_1,\cdots, c_n\otimes d_1,\cdots,c_1\otimes d_m,\cdots,c_n\otimes d_m)$.

Given a tree $T$, its classifying spaces $BT$ is obtained by considering the poset $E(T)$ of \textit{edges} of $T$ as a category, then the geometric realization of the nerve of this category is $BT$.

The paper~\cite{ShuffleTree} and the book~\cite{SandDHT} describe a diagram $\mathcal{I}_{sh}$ obtained by considering shuffles of $S,T$ and their intersections. 

\begin{lemma*}[{\cite[Proposition 5.1]{ShuffleTree}, \cite[Remark 5.2]{ShuffleTree}, and \cite[Proposition 4.10]{SandDHT}}] Consider trees $S,T$. Then, 
\begin{itemize}
\item $B(S)\times B(T)=colim_{A\in\mathcal{I}_{sh}}B(A)$,
\item $\Omega(S)\otimes_{BV}\Omega(T)=colim_{A\in\mathcal{I}_{sh}}\Omega(A)$.
\end{itemize}
\end{lemma*}

In the case of a reduced tree $T$, we studied $Poly(U(T))$ as a colimit of simplices over a diagram generated by certain products of \emph{linearlizations} of the poset $U(T)$.

If $U(T), U(S)$ are both chains, then to construct maximal operations in $\Omega(S)\otimes_{BV}\Omega(T)$ we consider an array where the $x$ axis is labeled by the operations in $T$ and the $y$ axis by the operations in $S$. Then the enumeration of paths from the left corner to the right corner, the shuffles, generate maximal operations. Note that the maximal simplices in the Minkovski sum $\{c,d\}(Poly(U(S)),Poly(U(T)))$ are created with the same rule, then 
$$d_{Poly\big(\{c,d\}\left(U(S),U(T)\right)\big) ,|U(S)|+|U(T)|}=\#\hbox{maximal operations in } \Omega(S)\otimes_{BV}\Omega(T),$$ (see~\cite[Example 4.4]{SandDHT}).

 If we have operations that satisfy the BV relation $$p\otimes d(c_1\otimes q,\cdots,c_n\otimes q)=\sigma_{n,m}^\ast\left((c\otimes q)(p\otimes d_1,\cdots,p\otimes d_m)\right),$$ we can construct instead a $n+m$ operation in two isomorphic ways according to the planar structure. 

Then
$d_{Poly\big(\{c,d\}\left(U(S),U(T)\right)\big) ,|P|-i}=\#\hbox{maximal operations in } \Omega(S)\otimes_{BV}\Omega(T),$ where we group $i$ times operations that satisfy the BV relation, identifying two functions if they differ by the choice of a plannar structure.

Let $T,S$ be reduced trees so that $U(T)=\chain[1]$ and $U(S)=\{c<d\}(\{c,d\}(\chain[1],\chain[1]),\chain[1])$.
Then the number of maximal simplices in $Poly(\{c,d\}(U(S),U(T)))$ is eight, while the number of shuffles in $B(T)\times B(S)$ is five, and from~\cite[Proposition 4.10]{SandDHT} the number of maximal operations in $\Omega(T)\otimes_{BV}\Omega(S)$ is five.

A dendroidal set is a functor $\Omega^{op}\rightarrow Set$.
In~\cite[Section 4.2]{SandDHT} there is a construction of a dendroidal set induced by a linear tree $P=[n]$,

$$\mathcal{E}(P):\Omega^{op}\rightarrow Sets$$ that sends $T\in \Omega$ into
$\hom_{\hbox{Weak poset}}({E}(T),P)$ where $E(T)$ is the poset of \textit{edges} of $T$.

\begin{remark}\label{RM:vectors}
Given a poset $P$, with $\mathcal{SH}(P)=\sum_{i=1}^n (-1)^{n-i}d_{P,i}\mathcal{SH}(i)$, we obtain the following information about the vectors $d_{P,i}$:
\begin{itemize}
\item When $P$ is the union of points, the vectors are related with Stirling numbers of the second kind $S(n,i)$ via $d_{P,i}=i!S(n,i)$~\cite[Theorem 5.6]{EulerandStirling}.
\item When $P$ is Wix\'arika poset, the vectors encode topological information (the Betti number) of the poset~\cite[Proposition 3.1, Proposition 3.3]{posets}. 
\item If $U(S), U(T)$ are chains, then 
$d_{Poly\big(\{c,d\}\left(U(S),U(T)\right)\big) ,|P|-i}=\#$ maximal operations/  planar structure in  $\Omega(S)\otimes_{BV}\Omega(T),$ where we group $i$ times operations that satisfy the BV relation.
\item The vector $d_{P,i}$ counts the number of surjective strict order preserving maps from $P$ to $\chain[i]$~\cite[Theorem 1]{beginning}.

\item For every tree $T$ and for every $n$, $$\#\mathcal{E}(\chain[n])_T=\sum_{i=1}^{|T|} (-1)^{|T|-i}d_{E(T),i}\multiset{n+1}{i}$$
where $E(T)$ is the poset of {edges} of $T$.

\item If $\mathcal{I}_P$ is the set constructed by the inclussion exclussion principle in Lemma~\ref{IEP}, then $d_{P,i}=\#\{I\subset\mathcal{I}_P| |I|=i\}.$ 

\item The polynomial  $\sum_{i=1}^{n}(-1)^{n-i}d_{P,i}\multiset{x+1}{i}$ evaluated at $n$ counts the number of shuffles between $P$ and a tree with $n$ vertices, that is, the number of maximal operations in the tensor product of $\Omega(P)$ and a linear tree with $n$ vertices.

\end{itemize}

\end{remark}

\section{Open questions}

\label{Sect6}
The set of virtual polytopes~\cite{VirtP} admits the structure of an algebra over the operad $\mathcal{SP}$. In that theory, the virtual polytope of an order polytope is the result of applying the operation that return the interior of the reflection around the origin to the polytope. In our interpretation, this assignment is an analog to the operadic version of Stanley reciprocity morphism on the Ehrhart series, which relates the Ehrhart series of a polytope, and the Ehrhart series of the interior of the polytope. Our definition of surjective weak maps suggests that there is a third geometric object missing in the theory of virtual polytopes, and in every place where order polynomials occur.

Given a poset $P$, in our study of order series we have information at two levels. One is the vector $\{d_{P,1},\cdots,d_{P,n}\}$, and second is the combinatorial meaning of the $n$-coefficient for $n\in\mathbb{N}$. They coincide because the action of the operad $\mathcal{FP}$ enumerates structures.
In the case of shuffle series we only proved that the action of the operad $\mathcal{SP}$ enumerates structures. In other words, we can define the action of the operad $\mathcal{FP}$ on shuffle series using the vectors $\{d_{P,1},\cdots,d_{P,n}\}$ but we were unable to prove that the resulting series counts shuffles. 
The first nontrivial case of a poset for which we do not have an explicit formula to compute its action on shuffle series is $N=\{x<y>z<w\}$. We do not know how to express this operation in terms of the Hadamard product and the Cauchy product of series. In Lemma 2.11 of \cite{APV} 
it is proved that the action on order series of the {4-ary} operation $N$ is not a lexicographic sum of the actions of $\{c<d\}$ and $\{c,d\}$ or any ternary operation. 

The vectors of $\zetafp[\{x<y>z<w\}]=5\zetafp[4]-5\zetafp[3]+\zetafp[2]$ should coincide with the vectors of $\mathcal{SH}(\{x<y>z<w\})$. Our experiments provide evidence of this assumption. For example, we obtain eight shuffles between the poset $N$ and the chain $\chain[1]$. In terms of the Hasse diagrams, there is a set of diagrams in which the points indexed by the second poset are both at the top, another diagram in which those points are at the bottom, and a diagram in which those points are in the middle; and the remaining five diagrams displayed below:

\begin{center}
\begin{tikzpicture}[scale=0.15]
\tikzstyle{every node}+=[inner sep=0pt]
\draw [black] (5,-12.8) circle (3);
\draw [black] (5,-12.8) circle (2.4);
\draw [black] (5,-3.6) circle (3);
\draw [black] (13.9,-13.7) circle (3);
\draw [black] (13.9,-3.4) circle (3);
\draw [black] (5,-22.4) circle (3);
\draw [black] (13.9,-22.4) circle (3);
\draw [black] (13.9,-22.4) circle (2.4);
\draw [black] (60.3,-23.8) circle (3);
\draw [black] (60.3,-13.7) circle (3);
\draw [black] (70.9,-13.7) circle (3);
\draw [black] (70.9,-13.7) circle (2.4);
\draw [black] (70.9,-23.8) circle (3);
\draw [black] (60.3,-3.4) circle (3);
\draw [black] (60.3,-3.4) circle (2.4);
\draw [black] (70.9,-3.4) circle (3);
\draw [black] (24.8,-3.4) circle (3);
\draw [black] (44.9,-2.8) circle (3);
\draw [black] (44.9,-2.8) circle (2.4);
\draw [black] (24.8,-23.8) circle (3);
\draw [black] (24.8,-23.8) circle (2.4);
\draw [black] (44.9,-23.8) circle (3);
\draw [black] (44.9,-13.7) circle (3);
\draw [black] (24.8,-13.7) circle (3);
\draw [black] (4.3,-31.8) circle (3);
\draw [black] (18.1,-31.8) circle (3);
\draw [black] (18.1,-31.8) circle (2.4);
\draw [black] (4.3,-41.1) circle (3);
\draw [black] (4.3,-41.1) circle (2.4);
\draw [black] (11.4,-41.1) circle (3);
\draw [black] (11.4,-41.1) circle (2.4);
\draw [black] (4.3,-53) circle (3);
\draw [black] (18.1,-41.1) circle (3);
\draw [black] (18.1,-53) circle (3);
\draw [black] (34.9,-13.7) circle (3);
\draw [black] (34.9,-13.7) circle (2.4);
\draw [black] (37.1,-53) circle (3);
\draw [black] (37.1,-53) circle (2.4);
\draw [black] (37.6,-41.7) circle (3);
\draw [black] (37.1,-31.7) circle (3);
\draw [black] (46.9,-41.7) circle (3);
\draw [black] (46.9,-41.7) circle (2.4);
\draw [black] (57.3,-53.4) circle (3);
\draw [black] (57.3,-42.3) circle (3);
\draw [black] (57.3,-42.3) circle (2.4);
\draw [black] (57.3,-31.7) circle (3);
\draw [black] (13.9,-10.7) -- (13.9,-6.4);
\fill [black] (13.9,-6.4) -- (13.4,-7.2) -- (14.4,-7.2);
\draw [black] (5,-19.4) -- (5,-15.8);
\fill [black] (5,-15.8) -- (4.5,-16.6) -- (5.5,-16.6);
\draw [black] (13.9,-19.4) -- (13.9,-16.7);
\fill [black] (13.9,-16.7) -- (13.4,-17.5) -- (14.4,-17.5);
\draw [black] (60.3,-20.8) -- (60.3,-16.7);
\fill [black] (60.3,-16.7) -- (59.8,-17.5) -- (60.8,-17.5);
\draw [black] (70.9,-20.8) -- (70.9,-16.7);
\fill [black] (70.9,-16.7) -- (70.4,-17.5) -- (71.4,-17.5);
\draw [black] (60.3,-10.7) -- (60.3,-6.4);
\fill [black] (60.3,-6.4) -- (59.8,-7.2) -- (60.8,-7.2);
\draw [black] (70.9,-10.7) -- (70.9,-6.4);
\fill [black] (70.9,-6.4) -- (70.4,-7.2) -- (71.4,-7.2);
\draw [black] (24.8,-10.7) -- (24.8,-6.4);
\fill [black] (24.8,-6.4) -- (24.3,-7.2) -- (25.3,-7.2);
\draw [black] (44.9,-20.8) -- (44.9,-16.7);
\fill [black] (44.9,-16.7) -- (44.4,-17.5) -- (45.4,-17.5);
\draw [black] (44.9,-10.7) -- (44.9,-5.8);
\fill [black] (44.9,-5.8) -- (44.4,-6.6) -- (45.4,-6.6);
\draw [black] (4.3,-50) -- (4.3,-44.1);
\fill [black] (4.3,-44.1) -- (3.8,-44.9) -- (4.8,-44.9);
\draw [black] (4.3,-38.1) -- (4.3,-34.8);
\fill [black] (4.3,-34.8) -- (3.8,-35.6) -- (4.8,-35.6);
\draw [black] (18.1,-38.1) -- (18.1,-34.8);
\fill [black] (18.1,-34.8) -- (17.6,-35.6) -- (18.6,-35.6);
\draw [black] (18.1,-50) -- (18.1,-44.1);
\fill [black] (18.1,-44.1) -- (17.6,-44.9) -- (18.6,-44.9);
\draw [black] (5,-9.8) -- (5,-6.6);
\fill [black] (5,-6.6) -- (4.5,-7.4) -- (5.5,-7.4);
\draw [black] (24.8,-20.8) -- (24.8,-16.7);
\fill [black] (24.8,-16.7) -- (24.3,-17.5) -- (25.3,-17.5);
\draw [black] (11.92,-11.45) -- (6.98,-5.85);
\fill [black] (6.98,-5.85) -- (7.14,-6.78) -- (7.89,-6.12);
\draw [black] (68.73,-21.73) -- (62.47,-15.77);
\fill [black] (62.47,-15.77) -- (62.71,-16.68) -- (63.4,-15.96);
\draw [black] (42.79,-21.67) -- (37.01,-15.83);
\fill [black] (37.01,-15.83) -- (37.22,-16.75) -- (37.93,-16.05);
\draw [black] (32.8,-11.56) -- (26.9,-5.54);
\fill [black] (26.9,-5.54) -- (27.1,-6.46) -- (27.82,-5.76);
\draw [black] (16.63,-50.39) -- (12.87,-43.71);
\fill [black] (12.87,-43.71) -- (12.83,-44.66) -- (13.7,-44.17);
\draw [black] (9.58,-38.72) -- (6.12,-34.18);
\fill [black] (6.12,-34.18) -- (6.21,-35.12) -- (7,-34.52);
\draw [black] (37.23,-50) -- (37.47,-44.7);
\fill [black] (37.47,-44.7) -- (36.93,-45.47) -- (37.93,-45.52);
\draw [black] (37.45,-38.7) -- (37.25,-34.7);
\fill [black] (37.25,-34.7) -- (36.79,-35.52) -- (37.79,-35.47);
\draw [black] (55.31,-51.16) -- (48.89,-43.94);
\fill [black] (48.89,-43.94) -- (49.05,-44.87) -- (49.8,-44.21);
\draw [black] (44.8,-39.56) -- (39.2,-33.84);
\fill [black] (39.2,-33.84) -- (39.4,-34.76) -- (40.12,-34.06);
\draw [black] (57.3,-50.4) -- (57.3,-45.3);
\fill [black] (57.3,-45.3) -- (56.8,-46.1) -- (57.8,-46.1);
\draw [black] (57.3,-39.3) -- (57.3,-34.7);
\fill [black] (57.3,-34.7) -- (56.8,-35.5) -- (57.8,-35.5);
\end{tikzpicture}
\end{center}

The assignment of SP-algebras $Posets\mapsto$ shuffle series is only surjective. In principle, this means that the action of the operad needs to take into account the label of the series. However, the action of $\widetilde{\{c,d\}}$ and $\widetilde{\{c<d\}}$ are independent of the poset labels, meaning that we can define those operations on any Maclaurin series with integer coefficients and constant-coefficient $1$. We wonder if this is true for every poset action. On this direction, consider a sequence  $P_i, Q_i$ so that $\mathcal{SH}(P_i)=\mathcal{SH}(Q_i)$. From~\cite[Theorem 4.6]{Doppelgangers2} and Theorem~\ref{main:thrm} we obtain the following property: for an arbitrary poset $R=(x_1,\cdots,x_n, \leq_R)$, $\widetilde{R}\left(\mathcal{SH}(P_1),\cdots,\mathcal{SH}(P_n)\right)=\widetilde{R}\left(\mathcal{SH}(Q_1),\cdots,\mathcal{SH}(Q_n)\right)$.

\section*{Acknowledgements}
The second author's research was supported by the National Research Foundation of Korea (NRF) grant funded by the Korean government (MSIT) (No. 2020R1C1C1A01008261) and (No. 2021R1A2B5B03087097). The Hasse diagrams were made using \url{http://madebyevan.com/fsm/}.

We thank the anonymous reviewers for their suggestions. We thank Antonio Arciniega-Nevárez and Anderson Vera for their feedback. We thank Ieke Moerdijk for his feedback on the first draft of this paper.

This work is dedicated to the memory of Professors Mykhaylo Shapiro Fishman, Valeri Kucherenko, and Steven Zelditch, whose teachings in various branches of analysis profoundly influenced the second author's understanding and approach to the subject.

\bibliographystyle{abbrv}
\bibliography{references}

\end{document}